\documentclass[journal,twocolumn,10pt,letter]{IEEEtran}

\IEEEoverridecommandlockouts
\ifCLASSINFOpdf
\else
\fi

\usepackage{bbold}
\usepackage{amsfonts}
\usepackage{nonfloat}
\usepackage{amssymb}
\usepackage[cmex10]{amsmath}
\usepackage{amsthm}
\usepackage{amssymb}
\usepackage{mathrsfs}
\usepackage{amsbsy}
\usepackage{setspace}
\usepackage{graphicx}
\usepackage{newclude}
\usepackage{latexsym}
\usepackage{lettrine} 

\usepackage{flushend}

\usepackage{scalerel,stackengine}
\stackMath

\usepackage{amsthm}

\newtheoremstyle{named}{}{}{\itshape}{}{\bfseries}{.}{.5em}{\thmnote{#3's }#1}
\theoremstyle{named}

\usepackage{amsthm}

\theoremstyle{definition}

\newtheorem{definition}{Definition}

\usepackage[table]{xcolor}
\usepackage{multirow}
\usepackage{rotating}
\usepackage{booktabs}
\usepackage{bbm} 
\usepackage{amsmath,epsfig,cite,amsfonts,psfrag}
\usepackage{lipsum}
\usepackage{cuted}
\usepackage{color}

\usepackage{epstopdf}
\usepackage{float}
\usepackage{mathtools}
\usepackage{bbm}
\usepackage{atbegshi}
\usepackage{balance}
\usepackage[utf8]{inputenc}
\usepackage{multirow}

\usepackage{accents}
\newlength{\dhatheight}

\allowdisplaybreaks

\newtheorem{lemma}{Lemma}

\graphicspath{{figures/}}
\allowdisplaybreaks

\makeatletter
\floatstyle{plain}
\newfloat{twocolequfloat}{b}{zzz}
\floatname{twocolequfloat}{Equation}
\newtheorem{Theorem}{Theorem}

\makeatother

\begin{document}



\title{ Precise Indoor Positioning:  Deployment of Reference Nodes to Guarantee LoS Condition
}
\title{Deployment of Reference Nodes to Guarantee a LoS Condition for Accurate Indoor Positioning}

    \author{\IEEEauthorblockN{Mohsen Abedi, Alexis A. Dowhuszko,~\IEEEmembership{Senior Member,~IEEE}, and Risto Wichman \\}
    \IEEEauthorblockA{Department of Information and Communications Engineering, 
    Aalto University, 02150 Espoo, Finland\\
    E-mails: \{mohsen.abedi, alexis.dowhuszko,  risto.wichman\}@aalto.fi} 
    }
\maketitle

\begin{abstract}
Accurate and precise positioning is required to guarantee the massive adoption of a wide range of 5G indoor applications, such as logistics and smart manufacturing. 
Native support for New Radio (NR) positioning services was included in 3GPP \mbox{Rel-16}, where angles-of-arrival/departure and time-(difference-)of-arrival measurements were specified in uplink and downlink.
However, all these positioning techniques assume Line-of-Sight~(LoS) propagation, 
suffering from systematic bias errors when such a condition cannot be guaranteed. To improve the accuracy and precision of indoor positioning systems that rely on proximity, triangulation, or trilateration principles, 
this paper considers the deployment of reference nodes to ensure LoS to one, two, or three nodes, respectively. For this purpose, the indoor service area is modeled with a graph whose nodes represent the polygons that partition the floor plan.  Then, the graph is partitioned into the minimum number of cliques, which specify the minimum number of reference nodes and their placement to guarantee a LoS condition regardless the user terminal position. The desired accuracy for positioning is guaranteed by setting a  minimum distance and a minimum angle between the reference nodes as two configuration parameters of the derived algorithms. 




\thispagestyle{empty}
\pagestyle{empty}
\end{abstract}

\vspace{0mm}


\vspace{-2mm}
\section{Introduction} 
\label{sec:1}
\vspace{-0.5mm}

Global Navigation Satellite Systems~(GNSS), such as GPS and/or Galileo, have been widely used so far to provide positioning services outdoors with the assistance of terrestrial mobile networks~\cite{pera2018,raza2018}. However, these satellite-based positioning systems cannot offer its services indoors, especially in relevant interior environments that have been identified in 5G standardization, such as offices, shops, warehouses, and hospitals~\cite{erik2021}. The key challenge of 5G-based positioning inside buildings is not necessarily the lack of enough communications infrastructure that could serve as reference nodes, which can be tackle in part with additional capital investments. 
The key challenge to enable 5G-based positioning in relevant use cases, such as the ones that exists in  Industry 4.0,  
are the hard indoor propagation characteristics in these situations~\cite{Tao2021}, 
with a simultaneous demand of improved positioning accuracy in the sub-decimeter range for smart industry applications~\cite{sail2021}.

With the evolution of 3GPP standards from 4G~(LTE) to 5G~(NR), network-based positioning has been enhanced in different
ways. Cell ID~(CID) was the baseline single-cell method defined in 3GPP Rel-8, in which the position of a User Equipment~(UE) was determined
based on the information that the serving cell had about the paging and tracking area update~\cite{raza2018}. Later on, 3GPP Rel-9 considered the support of Enhanced CID~(E-CID) with the addition of Received Signal Strength~(RSS), timing advance, and
Angle-of-Arrival~(AoA) measurements~\cite{pera2018}. After few incremental changes in the LTE-Advanced (LTE-A) releases, 3GPP Rel-16 introduced notable enhancements for multi-cell positioning methods for 5G (NR), targeting a less than $3$-meter accuracy indoors. For this, a new Positioning Reference Signal~(PRS) was defined to be used by various 5G positioning techniques, such as Round-Trip Time~(RTT), AoA, Angle-of-Departure~(AoD), and Time-Difference-of-Arrival~(TDOA)~\cite{keat2019}. Nevertheless, the positioning accuracy of these methods is notably affected by the existence of a Line-of-Sight~(LoS) condition~\cite{ren2021}. This is because Non-LoS~(NLoS) signals travel longer distances, introducing systematic errors that affect the UE position estimation.
Much effort was put to identify the existence of LoS and NLoS conditions in the received PRS, as well as in identifying ways to compensate the loss of accuracy in the latter case~\cite{xiao15,choi18}.

The performance of any mobile network-based positioning method depends on the strategy used to deploy the reference nodes that carry out the transmission and/or reception of the reference signals. Positioning accuracy is strongly influenced by the number and distribution of reference nodes~\cite{soga2011}.
The authors in \cite{posl2021} examine indoor positioning accuracy in a rectangular indoor area using grid and edge deployment strategies. Moreover, \cite{ryden2015baseline} discusses a regular deployment within a hall-shaped indoor environment.
In 5G networks, the positioning reference nodes are represented by Transmission and Reception Points (TRPs) associated with the gNBs (or 5G base stations), which transmit PRS in downlink and receive Sounding Reference Signals (SRS) in uplink. Then, different time- and angle-related parameters of these reference signals are reported to a 5G location server, which estimates the position of the UE~\cite{erik2021}. Unfortunately, 
when deploying these 3GPP nodes indoors, the presence of walls and obstacles may block the LoS components between TRP and UE~\cite{raza2018}. In the case, both time and angle measurements become biased, and the positioning uncertainty of indoor UEs is notably increased.  

From a more theoretical perspective, positioning techniques are broadly classified according to the fundamental principle that is used to estimate the UE position, namely: proximity, triangulation, and multilateration~\cite{pera2018}. The 3GPP (E-)CID method uses the proximity principle and associates the position of UE to the known one of the nearest reference node~\cite{soga2011}.  In contrast, the 3GPP positioning method based on AoA (AoD) measurements use triangulation, defining  the UE position as the intersection of the directions from where the reference signal is being received (transmitted)~\cite{brooks2005review}. Finally, 3GPP RTT, ToA, and TDoA methods are based on multilateration, assigning the position of the UE to the intersection of at least three geometrical figures; for RTT and ToA (TDoA), these are circles (hyperbolas) with center (focus) in each reference node~\cite{pera2018}. Thus, the deployment of positioning reference nodes to guarantee the LoS condition to one (proximity), two (triangulation) and three (multilateration) reference nodes is very importance to minimize  systematic errors and improve the accuracy of and Indoor Positioning System~(IPS). 

The use of 5G Millimeter-Wave (mmWave) frequency bands, with much wider bandwidths ($400$\,MHz) when compared to below $6$\,GHz bands ($100$\,MHz or less), provides a much better time resolution when computing propagation time (ToA and RTT) and time differences (TDoA) over-the-air~\cite{holm2020}. In addition, the use of higher frequency bands facilitates the packing of large-sized antenna arrays, enabling a finer resolution in the angular domain~\cite{shah2018}. 
The communication range of mmWave signals is expected to be as high as few hundred meters~\cite{rapp2013}, which is longer that the maximum distance between any pair of points in the indoor environments considered in 3GPP~\cite{38.901,Tao2021}.
 This makes the detailed planning of mmWave reference nodes to ensure LoS visibility in all points of the service area a range-unconstrained problem that resembles the well-known \emph{Art Gallery} prooblem~\cite{de1997computational}. 
 However, this is not the case when new 6G spectrum beyond mmWave is being considered~\cite{jian2021} since the communication range at these new frequency bands is notably reduced, creating a range-constrained version of the Art-Gallery problem~\cite{abedi2021visible}.

Although Terahertz (THz) band communications ($0.1$--$10$ THz) has the potential to alleviate the spectrum scarcity when moving from 5G to 6G~\cite{dowh2017}, the transmission distance is expected to be much shorter due to higher path loss attenuation~\cite{mold2014} and the difficulty to generate strong transmission power beyond $100$\,GHz~\cite{arms2012}. A similar range-constrained situation occurs when RF is replaced by optical wireless, such as Visible Light Communications~(VLC). However, the range of a VLC cell is not limited by the optical power of the transmitter, as Light-Emitting Diodes~(LEDs) are very efficient sources of optical power~\cite{dowh2022}, but rather by the directivity LED beams and the narrow Field-of-View~(FoV) of the Photodetectors~\cite{dowh2020}. Some simple examples of VLC-based indoor positioning applications include location-aware services, such as headsets for visitors in museums, 
asset tracking in hospitals (e.g., wheelchairs), airports (e.g., trolleys), and warehouses (e.g., consignments)~\cite{armstrong2013visible}. It is also possible to use VLC for indoor positioning in smart industrial environments, where the electric engines and machines of production lines create notable electromagnetic interference that can seriously affect the quality of communications over RF bands~\cite{sisi2018}.

Similar to the line-of-vision of the guards in the Art Gallery problem, mmWave, THz, and VLC systems are also blocked by walls and large-sized obstacles. However, the effective visibility range of a wireless system is different, as it depends on the link budget that is associated with
different frequency bands \cite{rapp2013,bottigliero2021low, mold2014}. 
For this purpose, we model an arbitrary floor plan using a graph that considers the range constraint of the air interface and the effect of physical obstacles in the LoS propagation. Then, the partitioning of this graph into the minimum number of cliques can be used to determine the minimum number of reference nodes and their placement to guarantee a LoS condition to all points of the service area, enabling an accurate positioning based on proximity.  
By adjusting the separation distance and angle between reference nodes for improved positioning precision, a modified version of this graph is partitioned into the minimum number of cliques. As a result of this graph partitioning, we can discover the minimum number of
reference nodes and their placement to ensure LoS conditions throughout the floor plan when using the other positioning principles, which are triangulation (two nodes) and multilateration (three or more nodes).

The rest of this paper is organized as follows: Section~\ref{sec:2} presents the system model and discusses the guidelines for deploying reference nodes, aiming at improving the accuracy of the three positioning principles. Section~\ref{sec:3} derives the algorithms to carry out the network element deployment when the LoS condition needs to be guaranteed to one (proximity), two (triangulation), and three (multilateration) reference nodes. Section~\ref{sec:4} carries out a detailed simulation analysis of the
performance of these deployment algorithms, showing the effect of the different configuration parameters on the detailed network planning. Finally, conclusions are drawn in Section~\ref{sec:5}.

\section{System Model and Problem Statement}
\label{sec:2}

\begin{figure}[!t]
\advance\leftskip -0.2cm
\hspace{0cm}\includegraphics[width=8.2cm]{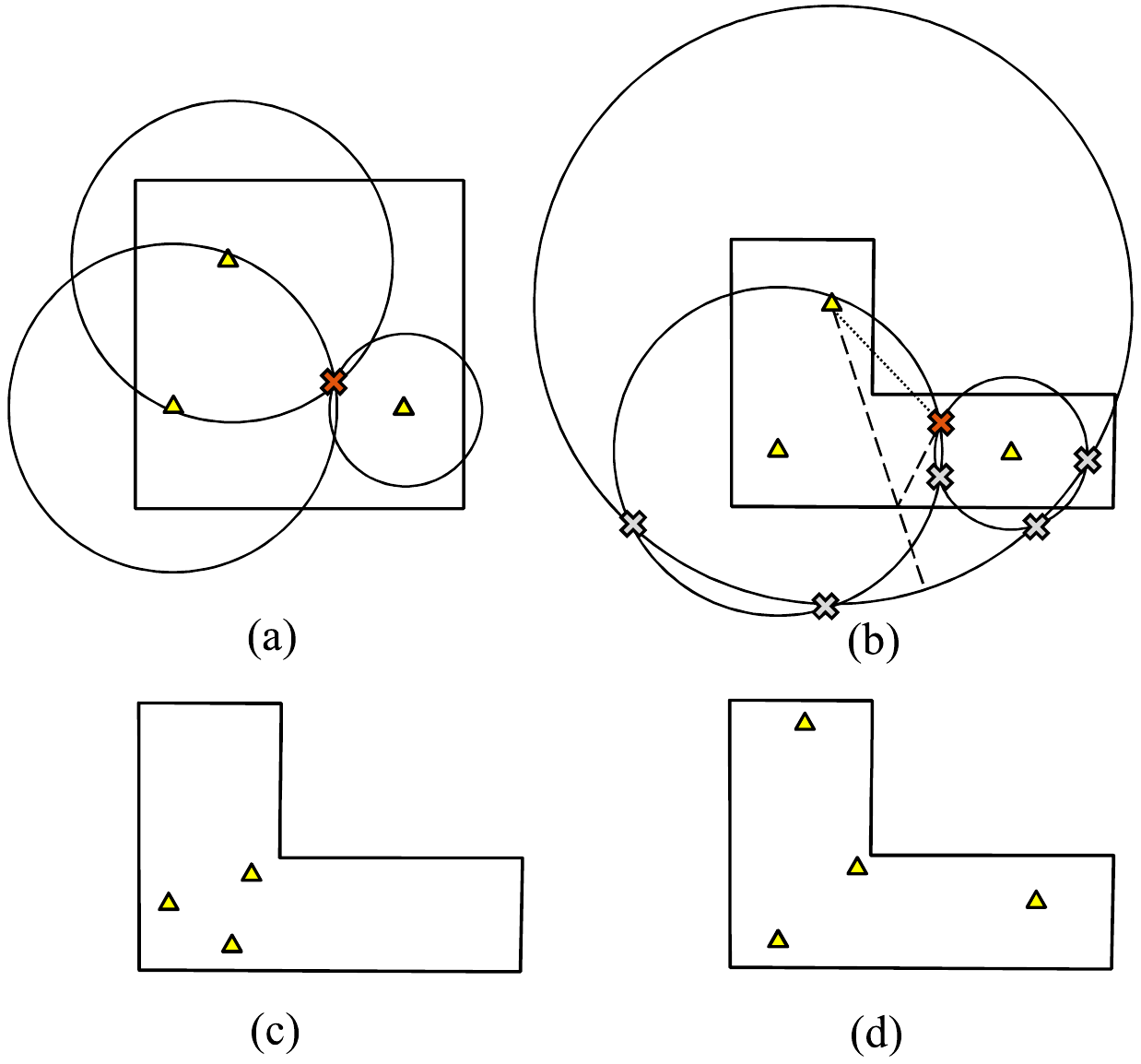}
    \vspace{-3.3cm}\caption{Effects of reference nodes placement (yellow triangles) over the accuracy of the positioning of the user equipment (red cross) based on the trilateration principle. a) Sample accurate (unbiased) positioning estimation in an square-shaped layout. b) Sample inaccurate (biased) positioning estimation (grey crosses) due to the blockage of the LoS link towards a reference node in an L-shaped layout. c-d) Sample reference node deployments that guarantee a LoS condition to all points in the service area of an L-shape layout. }\label{fig:Accuracy-precision}
\vspace{-2mm}
\end{figure}

The deployment of Positioning Reference Nodes (PRNs) is critical for enabling the precise and accurate positioning of a UE regardless its indoor location. We now examine how PRN placement affects indoor positioning accuracy and precision.

\subsection{Effects of PRN placement on the accuracy of IPS}
A PRN deployment that guarantees  LoS condition between a UE and a sufficient number of reference nodes, such as one PRN (proximity principle), two PRNs (triangulation principle), and three PRNs (trilateration principle), ensures accurate (unbiased) positioning.
For instance, in an IPS based on trilateration,
accurate positioning can only be achieved by taking at least three unbiased distance measurements, which will require LoS links between the UE and the three PRNs. 

As indicated in Fig. \ref{fig:Accuracy-precision}a),  a square-shaped layout with three PRNs (yellow triangle marks) ensures LoS conditions between a randomly placed UE (red cross mark) and the PRNs in the room, assuming an unlimited range of the wireless link used for ranging. 
Thus, this given PRN placement guarantees accurate positioning via trilateration for a UE regardless of its location.
 Fig. \ref{fig:Accuracy-precision}b) illustrates an L-shaped layout with some LoS links blocked (dotted line segment), resulting in overestimated (biased) distance measurements and thus multiple possible estimations of the location of the UE (gray cross marks).
 When PRNs are placed as in Fig.~\ref{fig:Accuracy-precision}c), the three LoS links that are needed for accurate positioning based on trilateration are ensured between the three PRNs and randomly placed UE, regardless its position in the layout. However, PRNs are relatively close and this may result in positioning uncertainty.
Finally, a four-PRN deployment, illustrated in Fig. \ref{fig:Accuracy-precision}d), ensures a LoS condition between at least three PRNs and a UE placed anywhere within the limits of the same L-shaped layout, contributing to accurate positioning in the whole service area as the trilateration nodes are better distributed in space.

In practice, PRNs are typically limited in service range, and indoor environments often have irregular shapes with walls that block the propagation of signals. Therefore, a systematic approach is required for an efficient deployment of PRNs to ensure LoS condition to a variable number of nodes throughout the entire layout.


\subsection{Effect of PRN deployment on the precision of IPS}\label{subsec:precision}

\begin{figure}[!t]
\centering
\advance\leftskip-0.1cm
\hspace{0.2cm}\includegraphics[width=9cm]{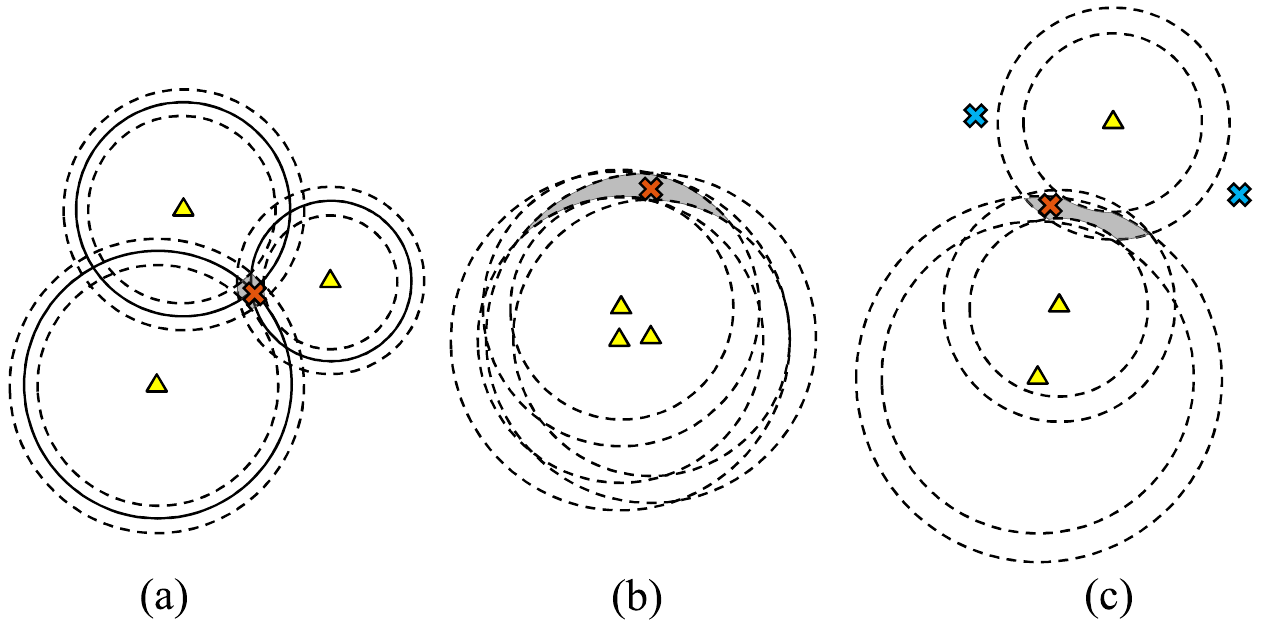}
    \vspace{-7.2cm}\caption{
     Effects of PRN placement over  positioning precision based on trilateration principle; a) An error-free distance measurement (solid circles) and a small positioning uncertainty due to noisy measurements (dashed circles), b) a tight PRN placement, c)  a co-linear PRN placement.  }\label{fig:Trilateration-Accuracy}
\vspace{-0.2cm}
\end{figure}

The uncertainty (error variance) around the estimated position of the UE as a result of erroneous measurements is an indicator of positioning precision. 
For instance in trilateration, 
 establishing LoS condition between a UE and three PRNs or more does not necessarily guarantee very precise positioning, although estimation will be unbiased.
 Figure ~\ref{fig:Trilateration-Accuracy}  visualizes how the placement of three PRNs affects the positioning uncertainty.
Fig.  \ref{fig:Trilateration-Accuracy}a) illustrates the precise positioning (solid circles) as well as errors in distance measurement (dashed circles) resulting in a small uncertainty in positioning (gray area).
 The positioning uncertainty may, however, be high if PRNs are closely located, as shown in Fig. \ref{fig:Trilateration-Accuracy}b). 
 Fig. \ref{fig:Trilateration-Accuracy}c) shows a vast positioning uncertainty due to the placement of co-linear PRNs.
In addition, it is still unclear whether this placement will resolve the positioning ambiguity since the location of both UEs (gray cross marks) are mirrored with respect to the line that connects the PRNs.
Therefore, positioning uncertainty is smallest as long as one of three visibility angles between UE and PRNs is close to $90$-deg. 
 To establish a relation between visibility angle of UEs and PRN deployments, we define a \emph{triplet} concerning two parameters, i.e., Minimum Separation Distance (MSD) and Minimum Separation Angle (MSA), denoted by $d_s$ and $\theta_s$, respectively.

\begin{figure}[!t]
\centering
\hspace{0.2cm}\includegraphics[width=8.2cm]{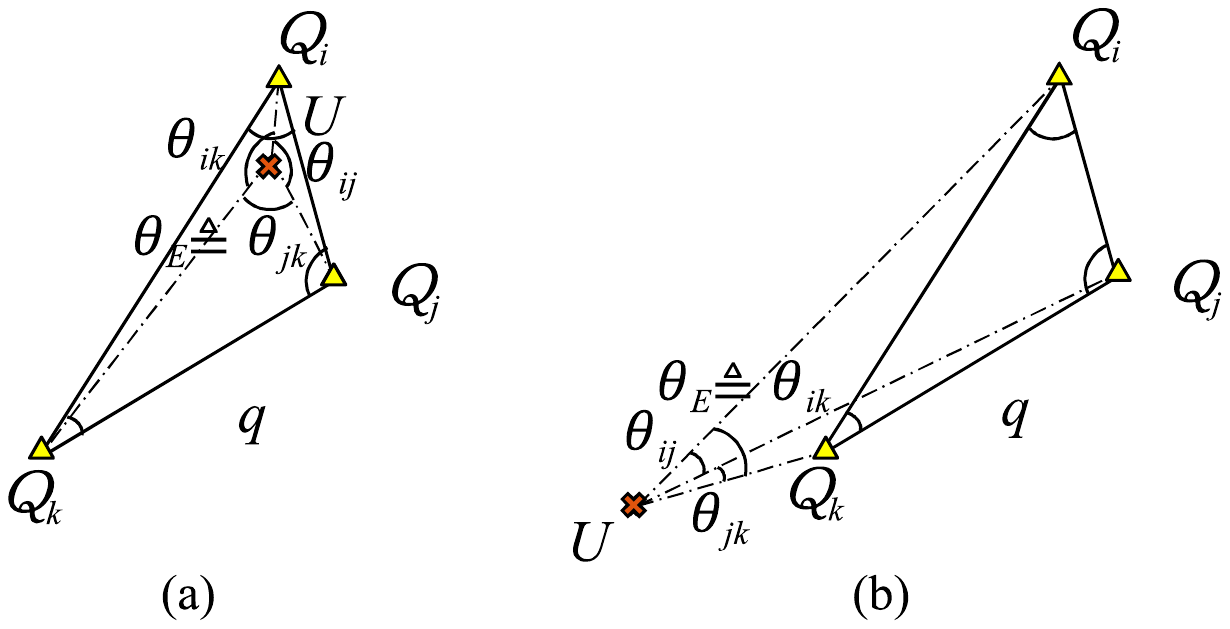}
    \vspace{-6.7cm}\caption{ The three visibility angles and identifying the EVA when the UE falls a) inside $q$ and b) outside $q$.  }\label{fig:angle_define}
\vspace{-0.5cm}
\end{figure}

\begin{definition}\label{well-separated.def} $\bf{Triplet~of~PRNs}$ is composed of three PRNs whose location points form the vertices of triangle $q$ such that
\begin{itemize}
    \item  the three sides of $q$ are greater than 
$d_s$ and,
    \item  the three inner angles of $q$ are greater than $\theta_s$,
\end{itemize}
by selecting the MSD and MSA values from the feasible intervals $0 \leq d_s \leq 2r$ and $0^\circ \leq \theta_s \leq  60^\circ$, respectively, where $r$ represent the maximum range.
\end{definition}
   Figure \ref{fig:angle_define} clarifies Definition \ref{well-separated.def}, where the vertices of $q$ are formed by three PRNs placed at the points $Q_i$, $Q_j$, and $Q_k$. 
 The set of PRNs is a triplet with $d_s$ and $\theta_s$, if $d_s\leq \|Q_iQ_j\|,\|Q_iQ_k\|,\|Q_jQ_k\|$ and $\theta_s\leq \hat{Q}_i,\hat{Q}_j,\hat{Q}_k$. 
Consider a UE at point $U$ with LoS conditions to the triplet.
  We define non-reflex visibility angles $\theta_{ij}\coloneqq  \widehat{Q_iUQ_j}$, $\theta_{ik}\coloneqq  \widehat{Q_iUQ_k}$, and $\theta_{jk}\coloneqq  \widehat{Q_jUQ_k}$,
 among which the Effective Visibility Angle (EVA), denoted by $\theta_E$, refers to the closest one to $90$ degrees.  We characterize how EVA is confined as a function of MSD and MSA parameters.
 \begin{Theorem}\label{Theorem:effective-angle}
      Assume a triplet of PRNs, forming a triangle $q$, have LoS conditions to a UE at $U$. Then, the EVA of the UE is bounded around the right angle by
\begin{equation} \label{eq1}
\begin{split}
&|90^\circ-\theta_E|\leq90^\circ- \theta_s,\\
\end{split}
\end{equation}
if $U$ lies inside $q$, and
\begin{equation} \label{eq:effective_LoS_angle_outside}
\begin{split}
&|90^\circ-\theta_E|\leq90^\circ- 2\times\arctan\big({\dfrac{d_s}{2r}}\tan{(\dfrac{\theta_s}{2}})\big),\\
\end{split}
\end{equation}
if $U$ lies outside $q$.
 \end{Theorem}



 \begin{proof}
     See Appendix A.
 \end{proof}

From Theorem  \ref{Theorem:effective-angle}, setting large MSD and MSA values for PRN deployment tasks leads to a UE's EVA anywhere in the layout being closer to a right angle.
However, this setting would require too many PRNs and might be infeasible due to physical restrictions.
Also, too small MSD and MSA values may result in an EVA far from the right angle as in Fig. \ref{fig:Trilateration-Accuracy}b) and c).  So, a PRN deployment that provides LoS conditions for all points within the layout and at least a triplet with suitable MSD and MSA values ensures improved positioning precision.

IPS based on AoA measures the angles between at least two PRNs and the UE, a principle known as triangulation.  Similarily, positioning uncertainty via triangulation might be large or unbounded due to errors in the angle measurement.
 We define a pair of PRNs as \emph{twin}, in case their distance is larger than an MSD value. Similarily, a PRN deployment requires to provides LoS to at least a twin  for every point in the layout to ensure precise positioning via triangulation.

\vspace{0mm}
\section{PRN deployment for Accurate and Precise Indoor Positioning}
\label{sec:3}

\vspace{0mm}

Let us start with the definition of $n$-LoS coverage of a deployment strategy of PRNs in a given layout, which identifies the event in which a UE located in any location within an indoor environment can experience a LoS conditions with at least $n$ PRNs within its service range.
 Authors in~\cite{abedi2021visible} give an overview of the essential definitions and the deployment approach to ensure 1-LoS coverage applicable to VLC networks. 
First, we shortly recall these definitions to maintain the coherence of the presentation.
Following this, we will discuss how these definitions can be used to deploy PRNs for 2-LoS and 3-LoS coverage.

 \begin{figure}[!t]
\centering
\hspace{0cm}\includegraphics[width=7.9cm]{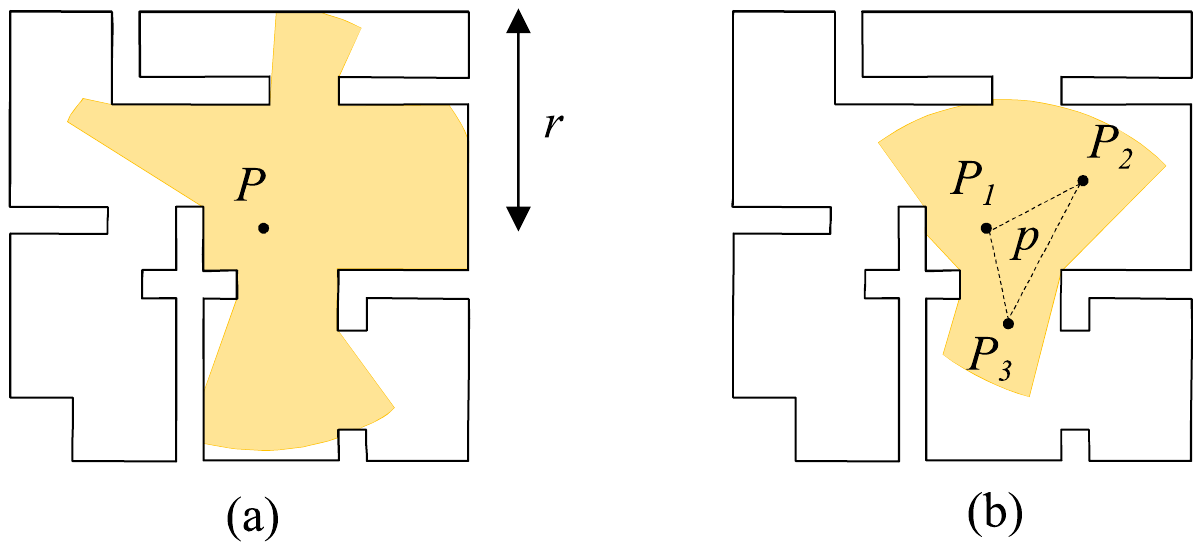}
    \vspace{-6.7cm}\caption{ LoS (visibility) area of a) the point $P$, and b) the triangle $p$. }\label{fig:LoS_areas}
\vspace{-0.4cm}
\end{figure}

Let $P$ be an arbitrary point, and $p$ be a polygon within a layout.
 Then, we introduce LoS area.
\begin{definition}\label{point_LoS.def}
  $\bf{LoS~area~of~a
 ~point}$ $\textbf{L}(P)$ refers to the locus $X$ such that the segment $XP$ entirely lies within the layout, and $\|XP\|\leq r$.
\end{definition}In other words, there is no layout edge between $P$ and the points in $\textbf{L}(P)$ while they are within the distance $r$ from $P$.
 
   \begin{definition}\label{polygon_LoS.def}
 $\bf{LoS~area~of~a~
 ~polygon}$ $\textbf{L}(p)$ refers to the set of all the points $X$ such that $X\in\textbf{L}(P_i)$ for $i=1, 2,..., l$, where $P_1, P_2, ..., P_l$ denote the vertices of $p$.
\end{definition}

It is immediate to conclude that  $\textbf{L}(p):=\bigcap_{i=1}^l\textbf{L}(P_i)$.
 Figure \ref{fig:LoS_areas} illustrates the LoS area of a point and a polygon in a  sample layout with a maximum range of $r$.
Note in Fig. \ref{fig:LoS_areas}b) that the vertex points $P_1$, $P_2$, and $P_3$ are not only openly visible, but they also fall within the distance $r$ from every point in $\textbf{L}(p)$. 

Every layout can be partitioned into triangles by adding a set of non-intersecting diagonals, known as \emph{triangulation partitioning}
, see \cite{de1997computational}.
However, smaller triangles are often required in order to obtain a better representation of a layout.

\begin{definition}\label{hyper-triangulation.def}$\bf{Hyper\,Triangulation}\,\mathcal{HT}(R)$uses the  triangulation partitioning as a basis and connects the midpoint of the largest side at every triangle, whose length exceeds $R$, to the opposite vertex ensuring no triangle side larger than $R$ left.
\end{definition}

\begin{figure}[!t]
\centering
\hspace{0.2cm}\includegraphics[width=8.2cm]{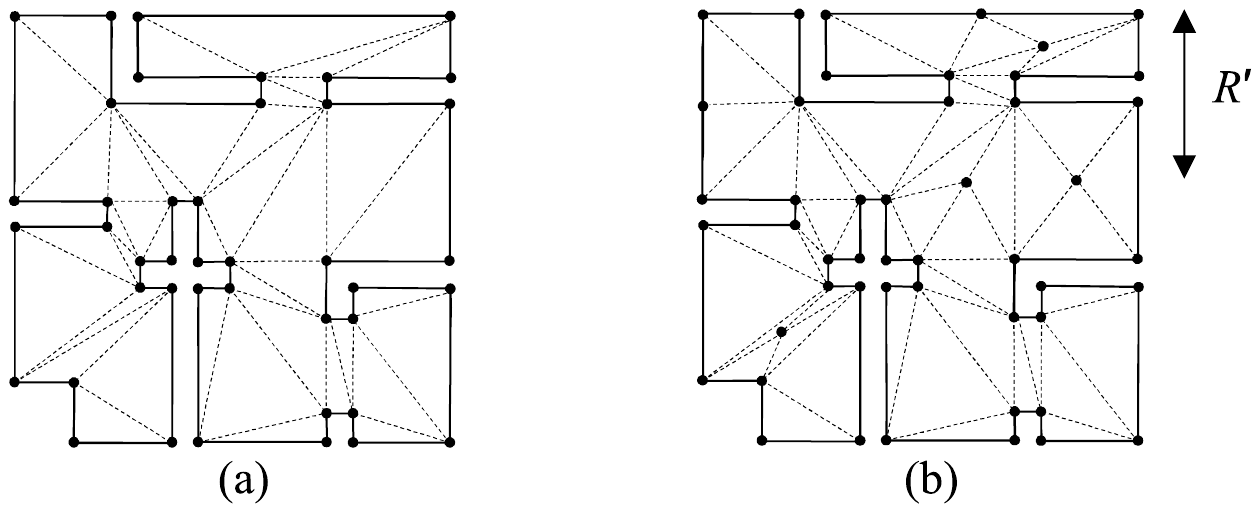}
    \vspace{-7.3cm}\caption{  Hyper triangulation by setting a)  $R=\infty$, and b) $R=R'$. }\label{fig:hyper_triangulation}
\vspace{-0.3cm}
\end{figure}

 Fig. \ref{fig:hyper_triangulation}a) shows triangulation or equivalently $\mathcal{HT}(R=\infty)$ for the sample layout, and Fig. \ref{fig:hyper_triangulation}b) represents $\mathcal{HT}(R=R')$, wherein no triangle side larger than $R'$ is left. 
In the following lemma, the significance of these definitions is illustrated.

\begin{lemma}\label{lemma:LoS}
    If $X\in \textbf{L}(q)$ for a point $X$, then $X \in \textbf{L}(Q)$, for any point $Q$ inside the polygon $q$. 
\end{lemma}
\begin{proof}
See \cite[Lemma 1]{abedi2021visible}.
\end{proof}
 This means that if the polygon's vertices in the layout have LoS conditions to a PRN, then all of the points inside the polygon do as well. Assume partitioning a layout into $M$ polygons
 wherein we deploy PRNs such that all the vertices of each polygon have LoS conditions commonly to a set of $n$  PRNs, e.g., a PRN for $n=1$, a twin   PRNs for $n=2$, or a triplet of PRNs for $n=3$. Then,  Lemma \ref{lemma:LoS} implies that this PRN deployment satisfies $n$-LoS coverage.
Next, we address PRN deployment for 1-LoS and 2-LoS coverage, which is then applied as an initial step to 3-LoS coverage to reduce the complexity of the algorithms.

\subsection{PRN deployment to ensure 1-LoS coverage}
Many indoor positioning systems use the proximity principle to identify Cell IDs for various purposes, including location-aware services and asset tracking, requiring an LoS condition between the UE and only one PRN.
Accordingly, we aim to determine the minimum number of PRNs and their locations to provide 1-LoS coverage, i.e., a UE anywhere in an arbitrary layout can establish an LoS condition to at least one PRN. 
To this aim, we construct an LoS graph to model the layout and introduce partitioning this graph into the minimum number of cliques as an equivalent optimization problem.

  \begin{figure}[!t]
\centering
\hspace{0.2cm}\includegraphics[width=8.5cm]{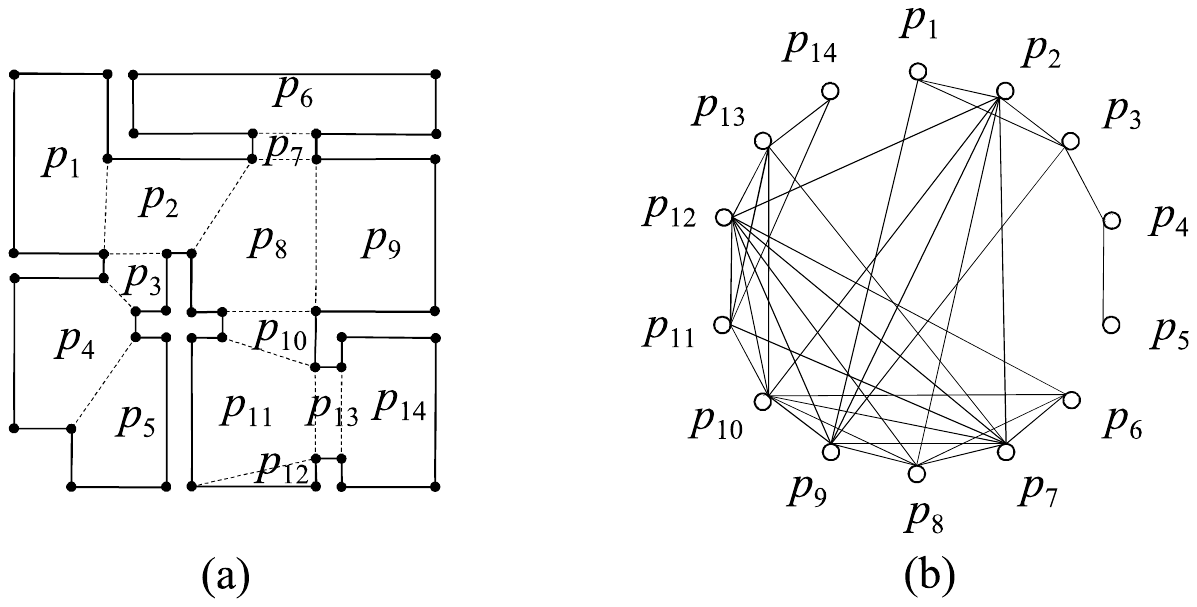}
    \vspace{-7.1cm}\caption{  a) The considered sample layout partitioned into $M=14$ convex polygons. b) The resulting primary LG with $r=\infty$. }\label{fig:PV_graph}
\vspace{-0.2cm}
\end{figure}

\begin{definition}\label{definition:PV-graph}$\hspace{-0.1cm}\bf{Primary~LoS}$ ${\bf{Graph}~(LG)}$, denoted by $\mathbb{G}^{(1)}$, refers to a simple unweighted  graph whose nodes represent the polygons (triangles)
$p_1$, ...., $p_M$ that partition the layout. 
A pair of nodes $p_i$  and $p_j$ is considered adjacent if their LoS areas overlap, i.e., $\textbf{L}(p_i)\cap \textbf{L}(p_j)\neq \emptyset$ for $1\leq i, j\leq M$ and $i\neq j$.
\end{definition}

PRN deployment begins with the creation of a primary LG.
Fig. \ref{fig:PV_graph}a) depicts partitioning of the sample layout into $M=14$ convex polygons and Fig. \ref{fig:PV_graph}b) shows the 
primary LG with $r=\infty$. In this case, since $\textbf{L}(p_2)\cap \textbf{L}(p_{10})\neq \emptyset$ and $\textbf{L}(p_2)\cap \textbf{L}(p_4)= \emptyset$, the nodes $p_{2}$ and  $p_{10}$ are adjacent, whereas the nodes $p_{2}$ and  $p_{4}$ are not adjacent.
 Suppose a set of nodes $q_1, q_2, ..., q_x$ forming a clique in a LG, say $C[q_1, q_2, ..., q_x]\subset \mathbb{G}^{(1)}$. Then, the LoS area of the clique is defined as $\textbf{L}(C):=\bigcap_{j=1 }^{x}\textbf{L}(q_j)$.

  \begin{figure}[!t]
\centering
\hspace{0.2cm}\includegraphics[width=8.2cm]{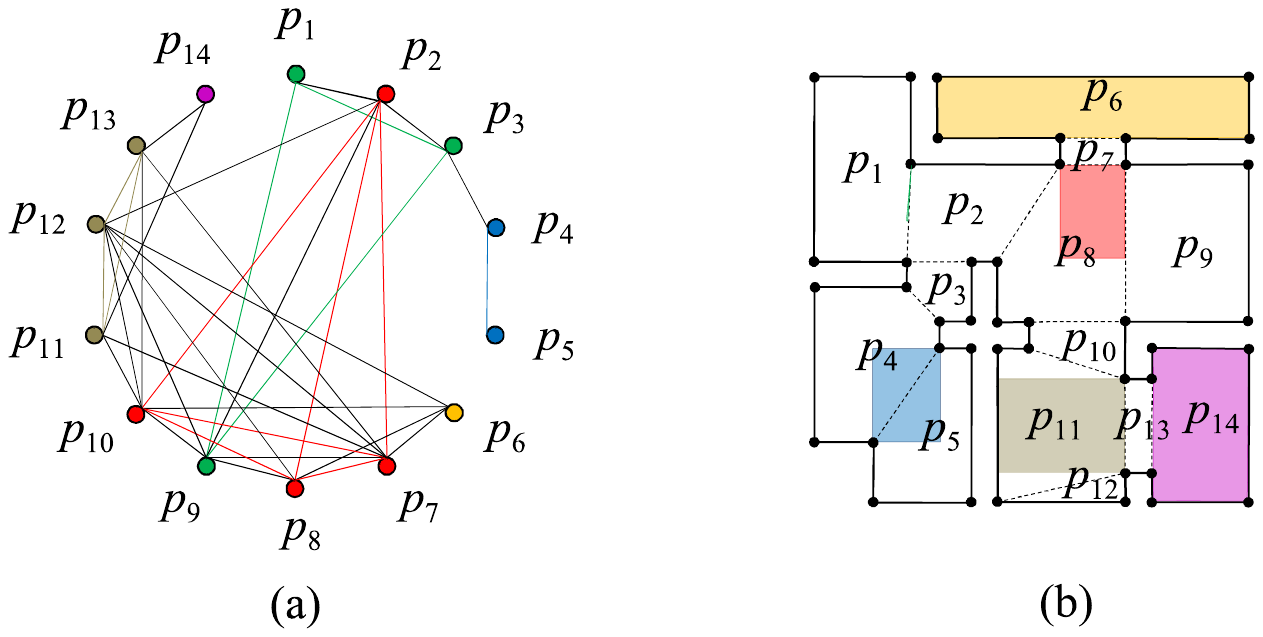}
    \vspace{-6.8cm}\caption{ A non-optimal number of PRNs and their placement areas to ensure 1-LoS coverage with $r=\infty$; a) A feasible clique partitioning of $\mathbb{G}^{(1)}$. b) The LoS areas of the cliques. }\label{fig:non-optimal}
\vspace{-0.2cm}
\end{figure}

 \begin{table}[t]
\caption{ Primary Maximal Clique  Clustering}\label{tabel:Max_Clique_partitioning}
\vspace{-0.3cm}
\begin{center}
\hspace{0cm}\begin{tabular}{ l | l}\hline\hline

{\footnotesize{~}}&
{\footnotesize{{\footnotesize{\textbf{Input:} ~~$\mathbb{G}^{(1)}[p_1,p_2,...,p_M]$ ~~~$\%$ The primary LG }}}} \\

\hline

{\footnotesize{1}}&
{\footnotesize{{\footnotesize{{$k\longleftarrow 0$ } }}}} \\

{\footnotesize{2}}&
{\footnotesize{{\footnotesize{{\bf{While}} ~~~$\mathbb{G}^{(1)}\neq \emptyset$: }}}} \\

{\footnotesize{3}}&{\footnotesize{{\footnotesize{~~~~$\mathbb{G}^{(1)}[q_1,q_2,..., q_M]\leftarrow$  $\texttt{Ascending-Sort}(\mathbb{G}^{(1)})$  }}}} \\

{\footnotesize{4}}&
{\footnotesize{{\footnotesize{{~~~~$k\longleftarrow k+1 $}}}}}\\

 {\footnotesize{5}}&
{\footnotesize{{\footnotesize{~~~~$C_k\longleftarrow \emptyset$}}}} \\

 {\footnotesize{6}}&{\footnotesize{{\footnotesize{~~~~{\bf{For}}~ $i$~ {\bf{from}} $1$ \bf{to} $M$: }}}} \\

 {\footnotesize{7}}&
{\footnotesize{{\footnotesize{~~~~~~~~{\bf{If}} ~$\texttt{Is\_Clique}(\mathbb{G}^{(1)},C_k\cup q_i)$==$1~ \&~ \textbf{L}(C_k\cup q_i)\neq \emptyset$: }}}} \\


{\footnotesize{8}} & {\footnotesize{~~~~~~~~~~~~~~~~~$C_k\longleftarrow  C_k \cup q_i$  }}\\ 


{\footnotesize{9}}&
{\footnotesize{{\footnotesize{~~~~$\mathbb{G}^{(1)}\longleftarrow \mathbb{G}^{(1)}-C_k~~~\%$ $C_k$: maximal clique }}}} \\
{\footnotesize{10}}&
{\footnotesize{{\footnotesize{~~~~$M\longleftarrow |\mathbb{G}^{(1)}|~~~~~~~~~~~\%$ The number of nodes in $\mathbb{G}^{(1)}$}}}} \\

{\footnotesize{11}}&
{\footnotesize{{\footnotesize{~~~~$C^{(1)}_k:=C_k$ $~~~~\%~$ The $k^{\text{th}}$ primary clique}}}} \\

{\footnotesize{12}}&
{\footnotesize{{\footnotesize{~~~~$\textbf{A}^{(1)}_k:=\textbf{L}(C^{(1)}_k)$$~~~~\%~$ The $k^{\text{th}}$ primary area}}}} \\

{\footnotesize{13}}&
{\footnotesize{{{\footnotesize{{\bf{Return}  }}}}}}\\

\hline

{\footnotesize{~}}&
{\footnotesize{{{\footnotesize{{\bf{Outputs:}~~ }    }}}}}\\

{\footnotesize{~}}&
{\footnotesize{{{\footnotesize{{$\mathcal{S}^{(1)}:=\{C^{(1)}_1, C^{(1)}_2,..., C^{(1)}_g\}$} ~$\%$ The set of primary cliques  }}}}}\\

{\footnotesize{~}}&
{\footnotesize{{\footnotesize{$\textbf{A}^{(1)}_1, \textbf{A}^{(1)}_2,..., \textbf{A}^{(1)}_g ~~~~~~~~~~~~~~~~\%$ The primary areas}}}} \\

\hline \hline
\end{tabular}
\end{center}
\vspace{-0.5cm}
\end{table}

\begin{lemma}\label{Lemma:Clique-partitioning}
With $g$ PRNs, each placed within the non-empty $\textbf{L}(C_1)$, $\textbf{L}(C_2)$, ..., and $\textbf{L}(C_g)$, 1-LoS coverage is ensured, in which the cliques $ C_1$, $ C_2$, ..., $ C_g$ partitions the nodes of $\mathbb{G}^{(1)}$.
\end{lemma}
\begin{proof}
Let $C_i[q_1, q_2,..., q_x]$ be a clique in $\mathbb{G}^{(1)}$ with $1\leq i\leq g$. From Lemma \ref{lemma:LoS}, a PRN located anywhere inside $\textbf{L}(C_i)$ covers the whole area of each  one of the polygons $q_1, q_2,..., q_x$.
Therefore, since the set of cliques $ C_1$, $ C_2$, ..., $ C_g$ covers all the nodes of $\mathbb{G}^{(1)}$, Lemma \ref{Lemma:Clique-partitioning} follows.
\end{proof}
Fig. \ref{fig:non-optimal}a) visualizes partitioning $\mathbb{G}^{(1)}$ into $g=6$ cliques, distinguished by $6$ different colors, whose LoS areas are given in Fig. \ref{fig:non-optimal}b) with the same color scheme.
 It is observed that placing a set of $g = 6$ PRNs, one lying inside each LoS area, ensures 1-LoS coverage of the sample layout. 
Nevertheless, we still require an algorithm for minimizing $g$.

 Lemma \ref{Lemma:Clique-partitioning} states that partitioning a primary LG into the minimum number of cliques results in the minimum number of PRNs necessary to achieve 1-LoS coverage.
Partitioning a graph into a minimum number of cliques is known as the \emph{minimum clique cover} problem in the literature \cite{espelage2001solve}. 
Algorithm \ref{tabel:Max_Clique_partitioning} presents the steps in partitioning the primary LG into a minimum number of cliques with non-empty LoS areas via the primary Maximal Clique Clustering MCC method.
Here, we first sort the nodes in an ascending degree order and choose the one with the minimum degree as a cluster. Then, we  add the rest of the nodes to this cluster in order as long as it forms a clique with non-empty LoS area. Finally, we remove this maximal clique from the primary LG and start over the algorithm 
until there are no nodes left.
The resulting maximal cliques are labeled as \emph{primary cliques} $C_i^{(1)}$ with the associated \emph{primary areas} $\textbf{A}_i^{(1)}:=\textbf{L}(C_i^{(1)})$ for $1 \leq i \leq g$.

 The steps of the primary MCC method are illustrated in Figure \ref{fig:optimal-single-coverage}. In Fig. \ref{fig:optimal-single-coverage}a), we select $p_5$ as the minimum degree node and add $p_4$ to form the maximal clique.  Removing this maximal clique, Fig. \ref{fig:optimal-single-coverage}b-d) display the remaining steps that end up with $g=4$ maximal cliques in Fig. \ref{fig:optimal-single-coverage}e), labeled as the primary cliques $C_1^{(1)}, C_2^{(1)}, C_3^{(1)}$ and $C_4^{(1)}$. So, taking into account the primary area  $\textbf{A}_i^{(1)}:= \textbf{L}(C_i^{(1)})$ for $1\leq i\leq g=4$, Fig. \ref{fig:optimal-single-coverage}f) illustrates that deploying the total number of four PRNs, anywhere inside each one of $\textbf{A}_1^{(1)},\textbf{A}_2^{(1)}, \textbf{A}_3^{(1)}$, and $\textbf{A}_4^{(1)}$, ensures 1-LoS coverage of the sample layout.

  \begin{figure}[!t]
\centering
\hspace{0.1cm}\includegraphics[width=9cm]{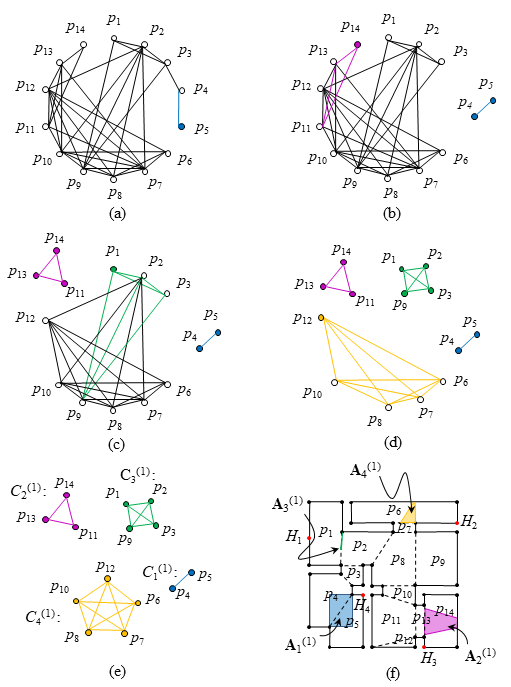}
    \vspace{-0.3cm}\caption{Determining the minimum number of PRNs and their placement  to ensure 1-LoS coverage  with $r=\infty$; a) the primary LG, a-e) the steps of primary MCC method, e) the primary cliques, and f) the primary areas. }\label{fig:optimal-single-coverage}
\vspace{-0.2cm}
\end{figure}

\begin{lemma}\label{Lemma:lower-bound}
Assume finding a set of $t$ points $H_1$, $H_2$,..., $H_t$ in a layout whose LoS areas are pairwise disjoint, i.e., $\textbf{L}(H_i)\cap \textbf{L}(H_j)= \emptyset$ for all $i\neq j$, known as a \emph{hidden} set of points. Then, $t$ reveals a lower bound to the minimum number of PRNs for a 1-LoS coverage, i.e., $t\leq g$.
\end{lemma}
\begin{proof}
 A 1-LoS coverage of the layout requires to cover at least the points $H_1, H_2,..., H_t$. So, we require to place at least a single PRN inside each one of the disjoint areas $\textbf{L}(H_1), \textbf{L}(H_2),..., \textbf{L}(H_t)$. Thus, Lemma \ref{Lemma:lower-bound} follows. 
\end{proof}

 Fig. \ref{fig:optimal-single-coverage}f) shows a hidden set of $t=4$ points $H_1$ to $H_4$, indicated by red dots. So from Lemma \ref{Lemma:lower-bound}, we have $t=4\leq g$. Therefore, the value $g=4$  derived by the primary MCC method in Figure \ref{fig:optimal-single-coverage} is optimal for a 1-LoS coverage.

\begin{table}[t]
\caption{ Edge Elimination}\label{tabel:Edge-elimination}
\vspace{-0.3cm}
\begin{center}
\hspace{0cm}\begin{tabular}{ l | l}\hline\hline

{\footnotesize{~}}&
{\footnotesize{{\footnotesize{\textbf{Inputs:}  {\footnotesize{{\footnotesize{$\mathbb{G}^{(1)}[p_1,p_2,...,p_M]$ ~~~~~~~~$\%$  The primary LG}}}}}}}} \\


{\footnotesize{~}}&
~~~~~~~~{\footnotesize{{\footnotesize{$\mathcal{S}^{(1)}, \textbf{A}^{(1)}_1,...,\textbf{A}^{(1)}_g$ ~~$\%$  Outputs of Algorithm \ref{tabel:Max_Clique_partitioning}  }}}} \\


\hline

{\footnotesize{1}}&
{\footnotesize{{\footnotesize{$\mathbb{G}^{(2)}\longleftarrow\mathbb{G}^{(1)}$  }}}} \\

{\footnotesize{2}}&
{\footnotesize{{\footnotesize{{\bf{For}} ~~$i$ ~{\bf{from}} $1$ 
{\bf{to}} $M$: }}}} \\

{\footnotesize{3}}&
{\footnotesize{{~~~\footnotesize{{\bf{For}} ~~$j$~ {\bf{from}} $1$ 
{\bf{to}} $M$: }}}} \\

{\footnotesize{4}}&
{\footnotesize{{\footnotesize{{~~~~~~\bf{If}}~ $p_i,p_j\in C_k^{(1)}$~$\&$~$\textbf{L}(p_i)\cap \textbf{L}(p_j)-\textbf{F}_{d_s}(\textbf{A}_k^{(1)})$==$\emptyset$: }}}} \\


{\footnotesize{5}}&
{\footnotesize{{\footnotesize{{~~~~~~~~~~~~~ \texttt{Edge-Eliminate}($\mathbb{G}^{(2)},p_i,p_j$) } }}}} \\

{\footnotesize{6}}&
{\footnotesize{{{\footnotesize{{\bf{Return} }  }}}}}\\

\hline

{\footnotesize{~}}&
{\footnotesize{{{\footnotesize{{\bf{Output:} ~ } {\footnotesize{{{\footnotesize{{$\mathbb{G}^{(2)}[p_1,p_2,...,p_M]$} ~~~~~$\%$ The secondary LG  }}}}}  }}}}}\\


\hline \hline

\end{tabular}
\end{center}
\vspace{-0.7cm}
\end{table}

\vspace{0mm}
\subsection{PRN deployment to ensure 2-LoS coverage}
\label{sec:double-coverage}
\vspace{0mm}
Given an MSD value, we determine the minimum number of PRNs and their placement  to ensure 2-LoS coverage, i.e., a UE anywhere in the layout has LoS conditions to at least a \emph{twin} PRNs. An IPS that applies the triangulation principle  represents an example wherein LoS conditions between a UE and two PRNs are necessary for positioning. A naive idea for  satisfying 2-LoS coverage is to place two PRNs inside each primary area, e.g., Fig. \ref{fig:optimal-single-coverage}f). However, such deployment might be infeasible if the given   
 MSD value is relatively large.

\begin{definition} 
\label{def:forbidden-region}$\hspace{-0.1cm}\bf{Forbidden ~region}$ $\textbf{F}_{d_s}(\textbf{A})$ refers to the locus whose distance to any point within the input area  $\textbf{A}$ is shorter than the MSD value, i.e., $\textbf{F}_{d_s}(\textbf{A}):=\bigcap_{O\in\textbf{A}} \{Q'\big{|} \|Q'O\| < d_s \}$.
\end{definition}

 From Definition  \ref{def:forbidden-region}, any circle with radius $d_s$ and its center inside $\textbf{A}$ covers the entire region $\textbf{F}_{d_s}(\textbf{A})$. The forbidden region is a self-inverse area-area mapping, i.e., $\textbf{F}_{d_s}(\textbf{F}_{d_s}(\textbf{A}))=\textbf{A}$ for a non-empty $\textbf{A}$. If $\textbf{A}$ is too large, we have $\textbf{F}_{d_s}(\textbf{A})=\emptyset$. Besides, the farthest point of a bounded area to any point on the plane coincides on the boundary. Thus, we can equivalently substitute $O\in \textbf{A}$ in Definition \ref{def:forbidden-region} with 
 $O\in \partial (\textbf{A})$, where $\partial(\textbf{A})$ denotes the boundary of $\textbf{A}$.
In case $\partial(\textbf{A})$ forms a polygon with straight edges, we can further limit the center points only to the vertices of $\textbf{A}$, i.e., $\textbf{F}_{d_s}(\textbf{A}):=\bigcap_{O\in v(\textbf{A})} \{Q'\big{|} \|Q'O\| < d_s \}$,  where $v(.)$ stands for the set of vertex points.

The primary areas and forbidden regions are key elements forming a set of twin PRNs. So, we modify the primary LG so that any pair of PRNs serving the same UE forms a twin. 
\begin{definition}
    \label{def:secondary PV graph}$\hspace{-0.1cm}$ $\textbf{Secondary}~\textbf{LoS~Graph}~ (\textbf{LG})$, denoted by ${\mathbb{G}^{(2)}}$,  refers to a modified version of the primary LG using Edge Elimination EE method by removing the edges whose
\begin{itemize}
    \item both ending nodes fall within a  common primary clique $C_k^{(1)}$ for some $1\leq k  \leq g$, and 
    \item the overlapping LoS area of the ending nodes entirely lies inside $\textbf{F}_{d_s}(\textbf{A}_k^{(1)})$.
\end{itemize}
\end{definition}
 To clarify the EE method in Algorithm  \ref{tabel:Edge-elimination}, the four primary areas in Fig. \ref{fig:optimal-single-coverage}f) are represented in Fig. \ref{fig:edgeremoval}a) whose forbidden regions are indicated by dotted patterns. Representing $\mathbb{G}^{(1)}$ in Fig. \ref{fig:edgeremoval}b), we eliminate three edges, shown by red dashed lines, to attain $\mathbb{G}^{(2)}$. As an example, considering the nodes $p_6$ and $p_{12}$ both within $C_4^{(1)}$, the area $\textbf{L}(p_6)\cap\textbf{L}(p_{12})$ totally lies inside  $\textbf{F}_{d_s}(\textbf{A}_4^{(1)})$, and thus the edge between $p_6$ and $p_{12}$ is eliminated. It is while the edge connecting $p_1$ and $p_9$ remains still, as both nodes are the components of $C_3^{(1)}$ and the area $\textbf{L}(p_1)\cap\textbf{L}(p_{9})-\textbf{F}(\textbf{A}_3^{(1)})$ is non-empty.  The secondary LG is a starting point in PRN deployment for 2-LoS coverage.

\begin{figure}[!t]
\centering
\hspace{-0.3cm}\includegraphics[width=9.1cm]{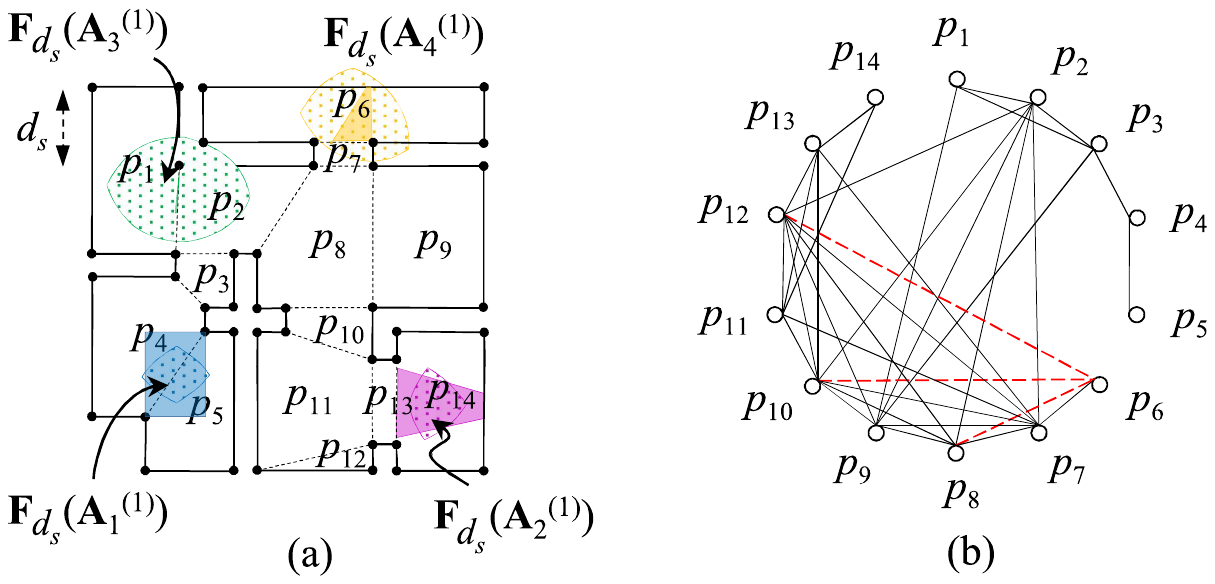}
    \vspace{-7.7cm}\caption{ Edge elimination method; a) Forbidden regions of the primary areas, and b) eliminated edges.}\label{fig:edgeremoval}
\vspace{-0.3cm}
\end{figure}

\begin{definition}\label{def:subset-clique-cover} $\hspace{-0.1cm}\bf{Clique~mapping}~$
$\mathcal{M}(\mathcal{S}^{(1)}, C)$ takes the set of primary cliques and an arbitrary clique $C\in \mathbb{G}^{(2)}$ as  inputs and returns all the primary areas whose associated primary cliques share at least one node with $C$, i.e., $\mathcal{M}(\mathcal{S}^{(1)}, C):=\{ \textbf{A}^{(1)}_{k} | C^{(1)}_k\in \mathcal{S}^{(1)}, C\cap C_{k}^{(1)}\neq\emptyset, 1\leq k \leq g \}$.
\end{definition}
This mapping  simplifies the following definition.
\begin{definition}\label{def:permitted-area} $\hspace{-0.1cm}\bf{Well\hspace{-0.05cm}-\hspace{-0.05cm}spaced~area}~$  $\textbf{W}_{d_s}(\mathcal{S}^{(1)},C)$  takes the set of primary cliques and an arbitrary clique $C\in \mathbb{G}^{(2)}$ as  inputs and returns the area that lies inside $\textbf{L}(C)$ while outside the union of forbidden regions for the primary  areas in $\mathcal{M}(\mathcal{S}^{(1)}, C)$, i.e.,  $\textbf{W}_{d_s}(\mathcal{S}^{(1)},C):=\textbf{L}(C)-\bigcup_{u\in\mathcal{M}(\mathcal{S}^{(1)}, C)}\textbf{F}_{d_s}(u)$.
\end{definition}

\begin{table}[t]
\caption{ Secondary Maximal Clique Clustering}\label{tabel:2-tier-Max_Clique_partitioning}
\vspace{-0.3cm}
\begin{center}
\hspace{0cm}\begin{tabular}{ l | l}\hline\hline

{\footnotesize{~}}&
{\footnotesize{{\footnotesize{\textbf{Inputs:} {\footnotesize{{\footnotesize{~$\mathcal{S}^{(1)}$ ~~~~~~~~~$\%$  Output of Algorithm \ref{tabel:Max_Clique_partitioning}  }}}} }}}} \\

{\footnotesize{~}}&
{\footnotesize{{~~~~~~~~~~~\footnotesize{$\mathbb{G}^{(2)}[p_1,p_2,...,p_M]$ ~~$\%$ Output of Algorithm \ref{tabel:Edge-elimination} }}}} \\



\hline

{\footnotesize{1}}&
{\footnotesize{{\footnotesize{$k\longleftarrow 0$ }}}} \\

{\footnotesize{2}}&
{\footnotesize{{\footnotesize{{\bf{While}} ~~~$\mathbb{G}^{(2)}\neq \emptyset$: }}}} \\

{\footnotesize{3}}&{\footnotesize{{\footnotesize{~~~~$\mathbb{G}^{(2)}[q_1,q_2,..., q_M]\longleftarrow$  $\texttt{Ascending\_Sort}(\mathbb{G}^{(2)})$  }}}} \\

{\footnotesize{4}}&
{\footnotesize{{\footnotesize{~~~~$k\longleftarrow k+1$}}}}  \\

 {\footnotesize{5}}&
{\footnotesize{{\footnotesize{~~~~$C_k\longleftarrow \emptyset$}}}} \\

 {\footnotesize{6}}&{\footnotesize{{\footnotesize{~~~~{\bf{For}}~ $i$ ~\bf{from} $1$ \bf{to} $M$: }}}} \\

 {\footnotesize{7}}&
{\footnotesize{{\footnotesize{~~~~~~~~{\bf{If}} ~~$\texttt{Is\_Clique}$ ($\mathbb{G}^{(2)}$, $C_k\cup q_i$) ==$1$: }}}} \\

{\footnotesize{8}}&
{\footnotesize{{\footnotesize{~~~~~~~~~~~{\bf{If}}~~  $\textbf{W}_{d_s}(\mathcal{S}^{(1)},C_k\cup q_i)\neq \emptyset~~ \% $ Definition \ref{def:permitted-area} }}}}  \\

{\footnotesize{9}} & {\footnotesize{~~~~~~~~~~~~~~~~$C_k\longleftarrow  C_k \cup q_i$  }}\\ 


{\footnotesize{10}}&
{\footnotesize{{\footnotesize{~~~~$\mathbb{G}^{(2)}\longleftarrow \mathbb{G}^{(2)}-C_k~~~\%$ $C_k$: The maximal clique }}}} \\
{\footnotesize{11}}&
{\footnotesize{{\footnotesize{~~~~$M\longleftarrow |\mathbb{G}^{(2)}|$, $C^{(2)}_k:= C_k~\% ~C^{(2)}_k\hspace{-1mm}:$ $k^{\text{th}}$ secondary clique }}}} \\


{\footnotesize{12}}&
{\footnotesize{{\footnotesize{~~~~$\textbf{A}^{(2)}_k:= \textbf{W}_{d_s}(\mathcal{S}^{(1)},C^{(2)}_k)~~~\%$ The $k^{\text{th}}$ secondary area}}}} \\ 




{\footnotesize{13}}&
{\footnotesize{{{\footnotesize{{\bf{Return}}   }}}}}\\

\hline

{\footnotesize{~}}&
{\footnotesize{{{\footnotesize{{\bf{Outputs:}}    }}}}}{\footnotesize{{{\footnotesize{{$\textbf{A}^{(2)}_1, \textbf{A}^{(2)}_2,..., \textbf{A}^{(2)}_{g'}~~~~~\%$  Secondary areas}}}}}}\\

{\footnotesize{~}}&
{\footnotesize{{{~~~~~~~~~~~\footnotesize{{ $\mathcal{S}^{(2)}:= \{C^{(2)}_1,..., C^{(2)}_{g'}\}$} $\%$ Set of secondary cliques   }}}}}\\


\hline\hline

\end{tabular}
\end{center}
\vspace{-0.6cm}
\end{table}

Placing a PRN within  $\textbf{W}_{d_s}(\mathcal{S}^{(1)},C)$ for a clique $C\in \mathbb{G}^{(2)}$ ensures a distance larger than $d_s$ to the PRNs placed in $\mathcal{M}(\mathcal{S}^{(1)}, C)$. Given the secondary LG, the primary cliques, and the MSD value, the \emph{secondary} MCC method  partitions $\mathbb{G}^{(2)}$ step by step into the minimum number of cliques $C^{(2)}_1, C^{(2)}_2,..., C^{(2)}_{g'}$, labeled as \emph{secondary cliques}, such that
 $\textbf{W}_{d_s}(\mathcal{S}^{(1)},C^{(2)}_{k})$, with $1 \leq k\leq g'$, are all non-empty. 
Then, we label them as \emph{secondary areas} $\textbf{A}^{(2)}_k:=\textbf{W}_{d_s}(\mathcal{S}^{(1)},C^{(2)}_{k})$ for $1\leq k\leq g'$. 
As given in Algorithm \ref{tabel:2-tier-Max_Clique_partitioning}, the secondary MCC method starts with sorting the nodes of $\mathbb{G}^{(2)}$ in an ascending degree order and adding the nodes to the first cluster $C^{(2)}_1$ as long as it forms a clique with a non-empty  $\textbf{W}_{d_s}(\mathcal{S}^{(1)},C^{(2)}_1)$. Then, we remove the maximal clique  from $\mathbb{G}^{(2)}$ and start over the algorithm until no node from the remaining graph left.

\begin{figure*}[!t]
\centering
\hspace{0.2cm}\includegraphics[width=16.2cm]{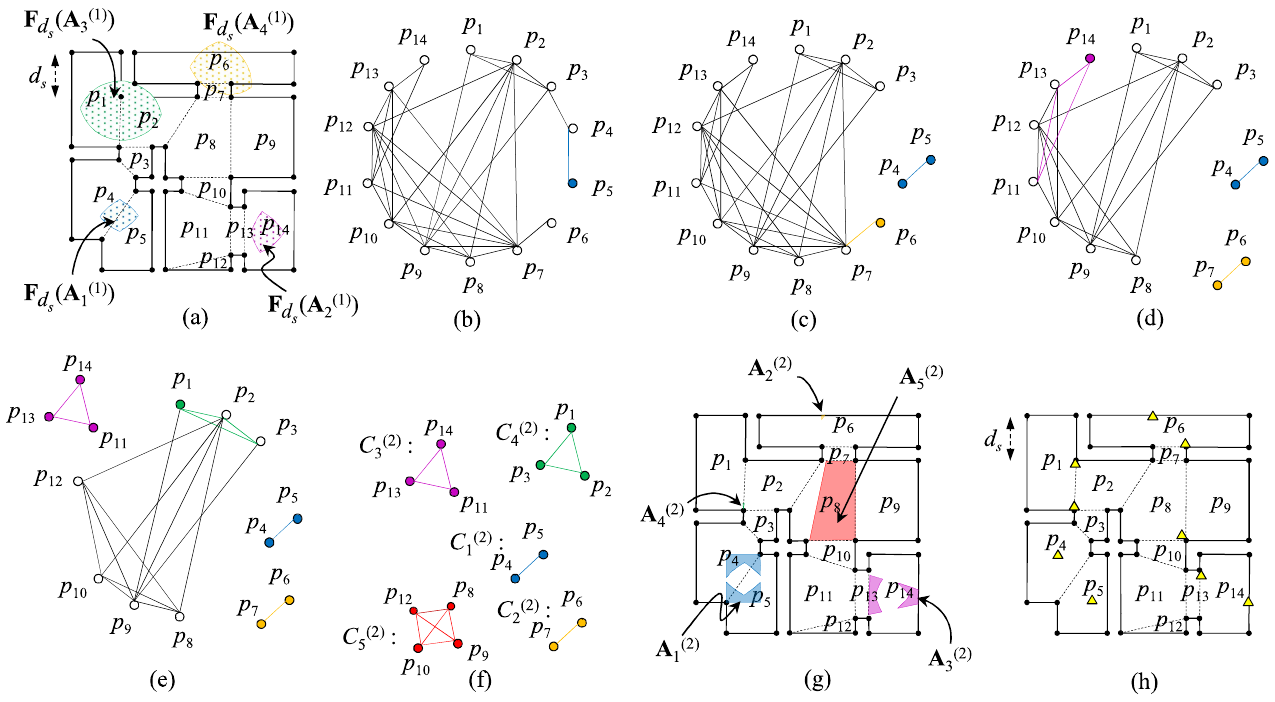}
    \vspace{-12.3cm}\caption{ Determining the minimum number of PRNs and their placement to satisfy 2-LoS  coverage with $r=\infty$ and the given MSD value $d_s$; a) Forbidden regions of the primary areas, b) the secondary LG, b-f) the steps of secondary MCC method, f) the resulting  $g'=5$ secondary cliques, g) the secondary areas, and h) the PRN deployment satisfying 2-LoS coverage.  }\label{fig:double-coverage-optimal}
\vspace{-0.2cm}
\end{figure*}

Figure \ref{fig:double-coverage-optimal} visualizes determining the minimum number of PRNs and their placement for 2-LoS coverage via the secondary MCC method with the given MSD value. Fig. \ref{fig:double-coverage-optimal}a) shows the forbidden regions of the four primary areas, and  Fig. \ref{fig:double-coverage-optimal}b) represents  $\mathbb{G}^{(2)}$ derived in Figure \ref{fig:edgeremoval}. The chain of Figs. \ref{fig:double-coverage-optimal}b-f) illustrates the steps of the secondary MCC method. In  Figs. \ref{fig:double-coverage-optimal}b), we select $p_5$ as the minimum degree node. Here,  the nodes  $p_4$ and $p_5$ form a maximal clique in $\mathbb{G}^{(2)}$ wherein both belong to the primary clique $C^{(1)}_1$, see Fig. \ref{fig:optimal-single-coverage}e). Thus, we form the maximal clique $ [p_5,p_4]$, labeled as the first secondary clique $C^{(2)}_1$, because we have $\mathcal{M}(\mathcal{S}^{(1)},C^{(2)}_1)=\{{\textbf{A}^{(1)}_1}\}$ and  $\textbf{W}_{d_s}(\mathcal{S}^{(1)},C^{(2)}_1)=\textbf{L}(C^{(2)}_1)- \textbf{F}_{d_s}(\textbf{A}^{(1)}_1)$ is non-empty. Taking similar steps, we have three secondary cliques in Fig. \ref{fig:double-coverage-optimal}e). Now, selecting $p_1$ as the minimum degree node, we form $C^{(2)}_4[p_1,p_2,p_3]$. Here, although the set of nodes $
[p_1,p_2,p_3,p_9]$ is also a clique in the remaining graph,  its well-spaced area is empty, so we avoid adding $p_9$ to $C^{(2)}_4$.  Finally, we end up with $g'=5$ secondary cliques $C^{(2)}_1, C^{(2)}_2,...,C^{(2)}_5$, as in Fig. \ref{fig:double-coverage-optimal}g), whose secondary areas $\textbf{A}^{(2)}_k:=\textbf{W}_{d_s}(\mathcal{S}^{(1)},C^{(2)}_{k})$ for $1\leq k\leq 5$ are shown in Fig. \ref{fig:double-coverage-optimal}h) with the same colors. Therefore, a total number of $g+g'=9$ PRNs is necessary and sufficient for 2-LoS coverage of the layout. Here, we place each PRN within the areas $\textbf{A}^{(1)}_1,..., \textbf{A}^{(1)}_4,  \textbf{A}^{(2)}_1,..., \textbf{A}^{(2)}_5$ as in Fig. \ref{fig:optimal-single-coverage}f) and Fig. \ref{fig:double-coverage-optimal}g),
such that the PRN associated to $\textbf{A}^{(2)}_k$ has the maximum distance to the PRNs placed within the areas in $\mathcal{M}(\mathcal{S}^{(1)}, C^{(2)}_k)$.
 It is because the PRNs associated with the areas  in $\mathcal{M}(\mathcal{S}^{(1)}, C^{(2)}_k)$  have partly the common areas to serve with the PRN for $\textbf{A}_k^{(2)}$. 
For instance, since  $\textbf{A}^{(1)}_1$ is the only element in  $\mathcal{M}(\mathcal{S}^{(1)}, C_1^{(2)})$, two PRNs are placed  on the corner points of $\textbf{A}^{(1)}_1$ and $\textbf{A}_1^{(2)}$, as in Fig. \ref{fig:double-coverage-optimal}h). Similarly since $\mathcal{M}(\mathcal{S}^{(1)}, C^{(2)}_5)=\{\textbf{A}^{(1)}_3, \textbf{A}^{(1)}_4\}$, we place the PRN inside $\textbf{A}^{(2)}_5$ on the corner point to have the maximum distance to the PRNs associated with $\textbf{A}^{(1)}_3$ and $ \textbf{A}^{(1)}_4$. Finally, the nine PRNs, placed as in Fig. \ref{fig:double-coverage-optimal}h), guarantee 2-LoS coverage of the sample layout.

It is worth mentioning that the minimum number of PRNs required for 2-LoS coverage is equal to or larger than two times the hidden set of points, i.e., $2t\leq g+g'$, see Lemma \ref{Lemma:lower-bound}. The equality may hold by setting small values for the MSD.

  \pagestyle{empty}

\begin{figure}[!t]
\centering
\hspace{0.2cm}\includegraphics[width=9cm]{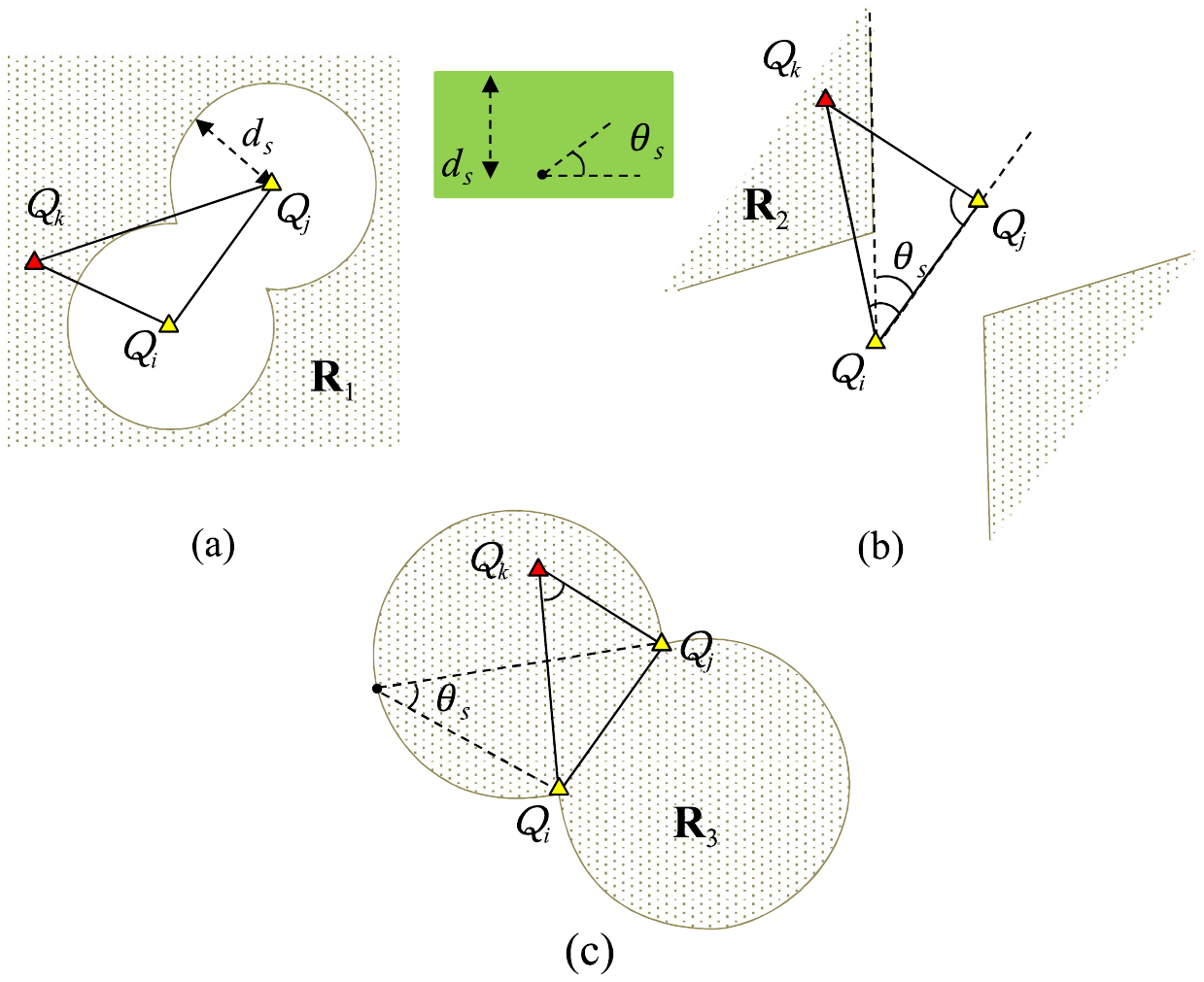}
    \vspace{-5.1cm}\caption{ Well-spaced area of the points $Q_i$ and $Q_j$ with respect to the given MSD and MSA values.}\label{fig:forbiden-triple}
\vspace{-0.5cm}
\end{figure}


\vspace{-1mm}
\subsection{PRN deployment to ensure 3-LoS coverage}
\vspace{-0.5mm}

Given MSD and MSA values, we  specify the minimum number of PRNs and their placement to ensure 3-LoS coverage, i.e., a UE anywhere in a layout has LoS conditions to at least a triplet of PRNs, see Definition \ref{well-separated.def}. 
An application requiring 3-LoS coverage is a ToA-based IPS using trilateration principle.
The deployment for 2-LoS coverage, such as the one in Fig. \ref{fig:double-coverage-optimal}h), is a starting point for 3-LoS coverage. We deploy a set of additional PRNs on top of the set already placed for 2-LoS coverage using the following definition.

\begin{table}[t]
\caption{ Trinary Maximal Clique Clustering}\label{tabel:trinary-Max_Clique_partitioning}
\vspace{-0.3cm}
\begin{center}
\hspace{0cm}\begin{tabular}{ l | l}\hline\hline

{\footnotesize{~}}&
{\footnotesize{{\footnotesize{\textbf{Inputs:}  ~~~ }}}} \\

{\footnotesize{~}}&
{\footnotesize{{\footnotesize{$\mathbb{G}^{(3)}[p'_1,p'_2,...,p'_{M'}]$ ~~$\%$ Definition \ref{def:trinary-graph} }}}} \\

{\footnotesize{~}}&
    {\footnotesize{{\footnotesize{$Q_1, Q_2,..., Q_{g+g'}$ ~~~~~$\%$ PRN locations for 2-LoS coverage  }}}} \\


\hline

{\footnotesize{1}}&
{\footnotesize{{\footnotesize{$k\longleftarrow 0$ }}}} \\

{\footnotesize{2}}&
{\footnotesize{{\footnotesize{{\bf{While}} ~~~$\mathbb{G}^{(3)}\neq \emptyset$: }}}} \\

{\footnotesize{3}}&{\footnotesize{{\footnotesize{~~~~$\mathbb{G}^{(3)}[q_1,q_2,..., q_{M'}]\longleftarrow$  $\texttt{Ascending\_Sort}(\mathbb{G}^{(3)})$  }}}} \\

{\footnotesize{4}}&
{\footnotesize{{\footnotesize{~~~~$k\longleftarrow k+1$}}}} \\

   {\footnotesize{5}}&
{\footnotesize{{\footnotesize{~~~~$C_k\longleftarrow \emptyset,~ \textbf{A}_k\longleftarrow \mathbb{R}^2$}}}} \\

 {\footnotesize{6}}&{\footnotesize{{\footnotesize{~~~~{\bf{For}}~ $i'$ ~{\bf{from}} $1$ {\bf{to}} $M'$: }}}} \\

 {\footnotesize{7}}&
{\footnotesize{{\footnotesize{~~~~~~~~{\bf{If}} ~~$\texttt{Is\_Clique}$ ($\mathbb{G}^{(3)}$, $C_k\cup q_{i'}$) ==$1$: }}}} \\

 {\footnotesize{8}}&{\footnotesize{{\footnotesize{~~~~~~~~~~~{\bf{For}}~ $i$ ~{\bf{from}} $1$ {\bf{to}} $g+g'$: }}}} \\

  {\footnotesize{9}}&{\footnotesize{{\footnotesize{~~~~~~~~~~~~~~{\bf{For}}~ $j$ ~{\bf{from}} $1$ {\bf{to}} $g+g'$: }}}} \\

{\footnotesize{10}}&
{\footnotesize{{\footnotesize{~~~~~~~~~~~~~~~~~{\bf{If}}~~  $Q_i,Q_j\in \textbf{L}(q_{i'}) ~\& ~ \|Q_iQ_j\| \geq d_s $:}}}}  \\

{\footnotesize{11}}&
{\footnotesize{{\footnotesize{~~~~~~~~~~~~~~~~~~~~{\bf{If}}~~  $ \mathbf{A}_k\cap \textbf{L}(q_{i'})\cap \textbf{W}_{d_s,\theta_s}(Q_i,Q_j)\neq \emptyset $}}}}  \\

{\footnotesize{12}} & {\footnotesize{~~~~~~~~~~~~~~~~~~~~~~~~$C_k\longleftarrow  C_k \cup q_{i'}$  }}\\ 


{\footnotesize{13}}&
{\footnotesize{{\footnotesize{~~~~~~~~~~~~~~~~~~~~~~~  $ \mathbf{A}_k\longleftarrow \mathbf{A}_k\cap \textbf{L}(q_{i'})\cap \textbf{W}_{d_s,\theta_s}(Q_i,Q_j)$}}}}  \\

{\footnotesize{14}}&
{\footnotesize{{\footnotesize{~~~~~~~~~~~~~~~~~~~~~~~ {\bf{Break}}~~~~$\%$ Return to the first {\bf{For}} loop }}}}    \\


{\footnotesize{16}}&
{\footnotesize{{\footnotesize{~~~~$\mathbb{G}^{(3)}\longleftarrow \mathbb{G}^{(3)}-C_k~~~\%$ $C_k$: The maximal clique }}}} \\
{\footnotesize{17}}&
{\footnotesize{{\footnotesize{~~~~$M'\longleftarrow |\mathbb{G}^{(3)}|$, $C^{(3)}_k:= C_k$, $\textbf{A}^{(3)}_k:= \textbf{A}_k$}}}} \\




{\footnotesize{18}}&
{\footnotesize{{{\footnotesize{{\bf{Return}}   }}}}}\\

\hline

{\footnotesize{~}}&
{\footnotesize{{{\footnotesize{{\bf{Output:}~~ }   }}}}}\\

{\footnotesize{~}}&
{\footnotesize{{{\footnotesize{{ $\mathcal{S}^{(3)}:= \{C^{(3)}_1, C^{(3)}_2,..., C^{(3)}_{g''}\}$} ~~$\%$ The set of trinary cliques  }}}}}\\

{\footnotesize{~}}&
{\footnotesize{{{\footnotesize{{ $\textbf{A}^{(3)}_1, \textbf{A}^{(3)}_2,..., \textbf{A}^{(3)}_{g''}~~~~\%$  Trinary areas}}}}}}\\
\hline \hline

\end{tabular}
\end{center}
\vspace{-0.5cm}
\end{table}

\begin{definition} 
\label{def:well-spaced-area-two-points}$\hspace{-0.1cm}{\bf{Well\hspace{-0.1cm}-\hspace{-0.1cm}spaced~area}}~\textbf{W}_{d_s,\theta_s}(Q_i,Q_j)$ refers to the locus $Q_k$ that forms a triplet with a twin PRNs located at   $Q_i$ and $Q_j$.
\end{definition}
In other words,  $\textbf{W}_{d_s,\theta_s}(Q_i,Q_j)$ is   the locus of $Q_k$ such that $d_s\leq \|Q_iQ_k\|,\|Q_jQ_k\|$ and $\theta_s\leq\hat{Q}_i, \hat{Q}_j, \hat{Q}_k$ for the triangle $q=:\triangle Q_iQ_jQ_k$. Figure \ref{fig:forbiden-triple} illustrates the well-spaced area of two given points $Q_i$ and $Q_j$. Selecting $Q_k$ within $\mathcal{R}_1$, ensures the condition  $d_s\leq \|Q_iQ_k\|,\|Q_jQ_k\|$, where $\mathcal{R}_1$ denotes the exterior of the circles  with the center points $Q_i$ and $Q_j$ and radius $d_s$, see  Fig. \ref{fig:forbiden-triple}a). Besides, denoting the shaded region in Fig. \ref{fig:forbiden-triple}b) by $\mathcal{R}_2$,  choosing $Q_k$ within $\mathcal{R}_2$ ensures $\theta_s\leq\hat{Q}_i, \hat{Q}_j$. As in Fig. \ref{fig:forbiden-triple}c), we draw two circles that take the segment $Q_iQ_j$ as a chord with radius $\|Q_iQ_j\|/ 2\sin{\theta_s}$, where $\mathcal{R}_3$ denotes the inner space. Therefore, falling $\hat{Q}_k$ inside $\mathcal{R}_3$ ensures $\theta_s\leq \hat{Q}_k$. Finally, in Figure  \ref{fig:forbiden-triple}, we have 
$\textbf{W}_{d_s,\theta_s}(Q_i,Q_j):=\mathcal{R}_1 \cap \mathcal{R}_2 \cap \mathcal{R}_3$.


\begin{figure*}[!t]
\centering
\hspace{0.2cm}\includegraphics[width=16.5cm]{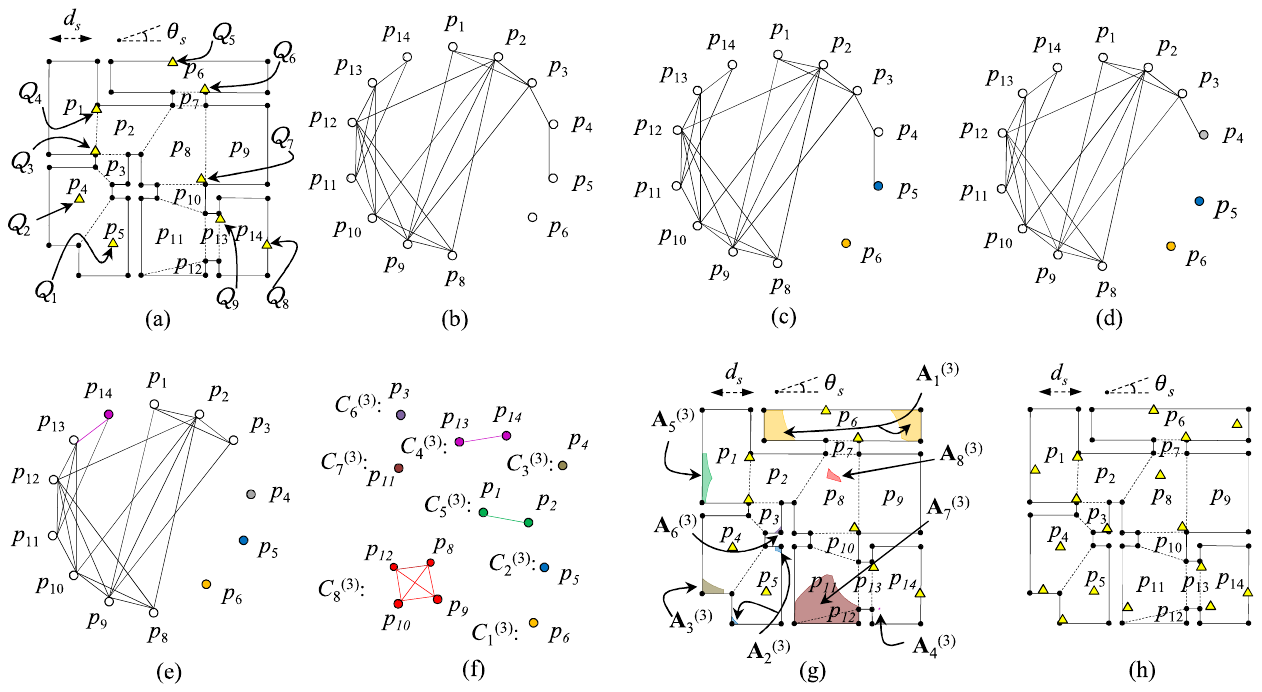}
    \vspace{-12.6cm}\caption{ Determining the minimum number of PRNs and their placement to achieve  3-LoS coverage with $r=\infty$ and the given MSD  and  MSA values; a)  The deployment of Fig. \ref{fig:double-coverage-optimal}h, b) the trinary LG, c-f) the steps of the trinary MCC method, f) the resulting  $g''=8$ trinary cliques, g) the trinary areas, and h) the PRN deployment.  }\label{fig:triple-coverage-optimal}
\vspace{-0.1cm}
\end{figure*}

\begin{definition}\label{def:trinary-graph}$\hspace{-0.1cm}{\bf{Trinary~LoS~Graph~(LG)}}$, denoted by $\mathbb{G}^{(3)}$,  is derived by  removing the nodes in the secondary LG whose LoS area contains at least a triplet of PRNs already placed for 2-LoS coverage.
\end{definition}

By creating the trinary LG, we  increase the chance that the resulting number of cliques in the clique clustering process is the minimum.
Figure \ref{fig:triple-coverage-optimal} depicts the steps of the trinary MCC method to determine the minimum number of PRNs and their placement for a 3-LoS coverage. 
Fig. \ref{fig:triple-coverage-optimal}a) displays the set of PRNs already placed for 2-LoS coverage as in Fig. \ref{fig:double-coverage-optimal}h) while  Fig. \ref{fig:triple-coverage-optimal}b) represents $\mathbb{G}^{(3)}$. 
Here, since the PRNs located at $Q_5$, $Q_6$, and $Q_7$ form a triplet within $\textbf{L}(p_7)$, node $p_7$ is removed. 
To find the minimum number of cliques that partition $\mathbb{G}^{(3)}$, Figs. \ref{fig:triple-coverage-optimal}c-f) illustrates the steps of the trinary MCC method as in Algorithm \ref{tabel:trinary-Max_Clique_partitioning}. The single node $p_6$ is solely clustered as a maximal clique, labeled as the first trinary clique $C_1^{(3)}[p_6]$. Here, since points $Q_5$ and $Q_6$ fall within $\textbf{L}(p_6)$, the first trinary area is derived as $\textbf{A}_1^{(3)}:=\textbf{L}(C_1^{(3)})\cap \textbf{W}_{d_s,\theta_s}(Q_5,Q_6)$. We cluster the next minimum degree node $p_5$, solely as the second trinary clique $C_2^{(3)}[p_5]$, see Fig. \ref{fig:triple-coverage-optimal}c). Here, although $p_5$ can form a clique with $p_4$ containing the twin PRNs located at $Q_1$ and $Q_2$ in their LoS area, we avoid adding $p_4$ to $C_2^{(3)}$ because $\textbf{L}(p_5)\cap \textbf{L}(p_4)\cap\textbf{W}_{d_s,\theta_s}(Q_1,Q_2)=\emptyset$.  Continuing this procedure, we end up with $g''=8$ trinary cliques in Fig. \ref{fig:triple-coverage-optimal}f)  whose trinary areas are displayed in Fig. \ref{fig:triple-coverage-optimal}g).  For instance, the twin PRNs located at $Q_4$ and $Q_7$ fall inside $\textbf{L}(p_9)$, while the twin PRNs located at $Q_6$ and $Q_7$ lie inside $\textbf{L}(p_8)\cap \textbf{L}(p_{10})\cap \textbf{L}(p_{12})$. Therefore, forming the last trinary clique $C_8^{(3)}[p_8,p_9,p_{10},p_{12}]$, the last trinary area is defined as $\textbf{A}_8^{(3)}=\textbf{L}(C_8^{(3)})\cap \textbf{W}_{d_s,\theta_s}(Q_6,Q_7)\cap  \textbf{W}_{d_s,\theta_s}(Q_4,Q_7)$. Finally, placing a set of additional PRNs, one anywhere inside  each trinary area, results in the total number of $g+g'+g''=17$ PRNs, ensuring the 3-LoS coverage of the given layout. 

Notably, the minimum number of PRNs required for 3-LoS coverage is lower-bounded by three times the  hidden set of points, i.e., $3t\leq g+g'+g''$, see Lemma \ref{Lemma:lower-bound}. The equality may hold by setting small values for the MSD and MSA values.

\vspace{0mm}
\section{Simulation results}
\label{sec:4}
\vspace{-0mm}

We evaluate the capability of the proposed algorithms  to satisfy 1-LoS, 2-LoS, and 3-LoS coverage in the $22m \times 22m$ sample layout.  
For  $r=\infty$, $r=6m$, and $r=3m$,  we partition the layout by  $\mathcal{HT}(R=3m), \mathcal{HT}(R=2m)$, and $\mathcal{HT}(R=1.2m)$, resulting in $M=387, M=762$, and $M=1919$ triangles, respectively,   representing the number of nodes in the primary LG.

\begin{figure*}[!t]
\advance\leftskip+2.8cm
\includegraphics[width=13.1cm]{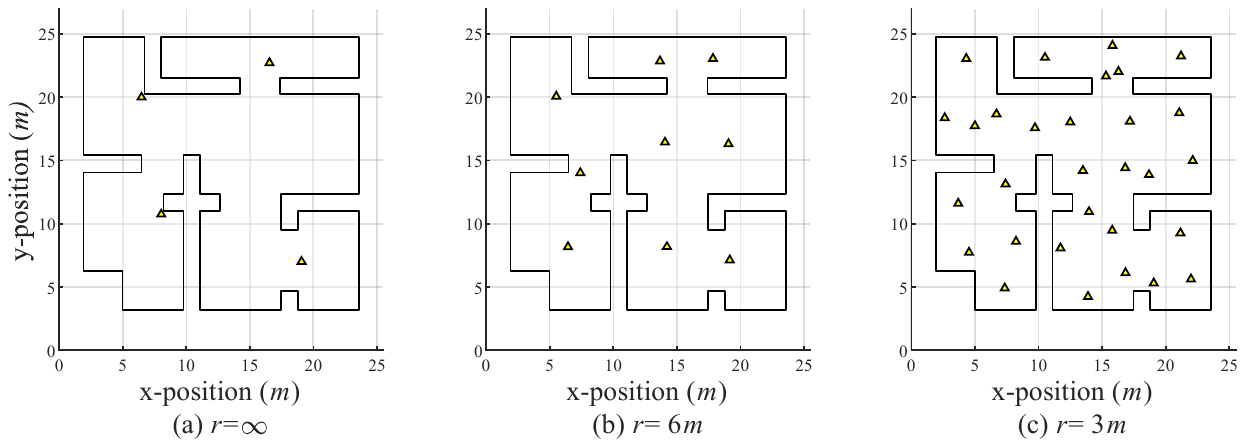}
    \vspace{-12.7cm}\caption{ Minimum number of PRNs and their placements for 1-LoS coverage}\label{fig:single-coverage}\vspace{-0.1cm}
\end{figure*}

  Figure \ref{fig:single-coverage} shows the minimum number of PRNs and their locations to achieve 1-LoS coverage using the primary MCC method.
Based on the representation in Fig. \ref{fig:single-coverage}a), placing only $4$ PRNs ensures that any UE located anywhere within the  layout has LoS condition to least one PRN with an unlimited maximum range.
By setting $r = 6$ and $3m$ as in Fig. \ref{fig:single-coverage}b) and c),
 we achieve 1-LoS coverage through 9 and 30 PRNs, respectively.
It is observed that a smaller range requires a higher number of PRNs to achieve the same performance. 

 \begin{figure*}[!t]
\advance\leftskip+2.5cm
\includegraphics[width=13.2cm]{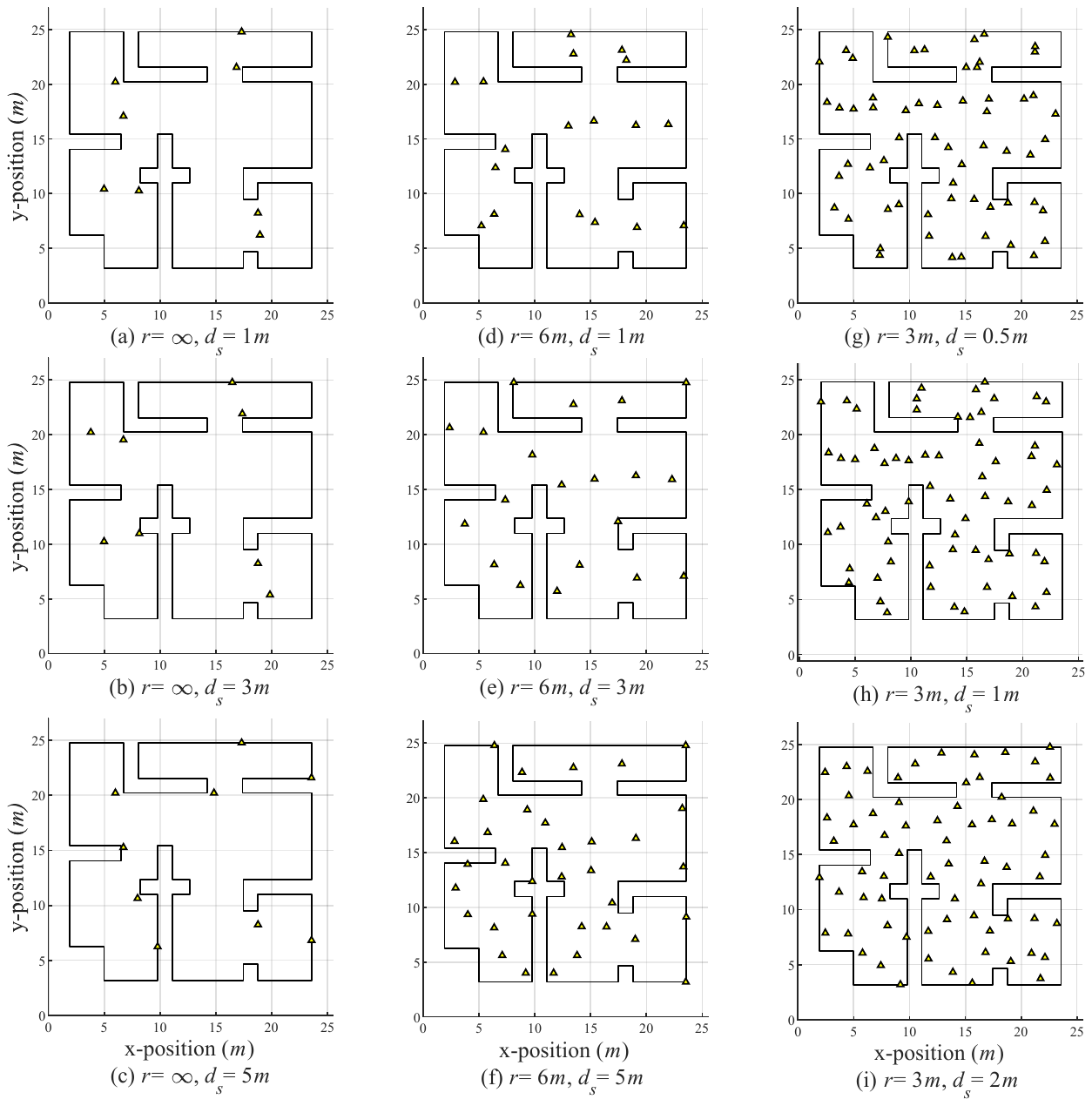}
    \vspace{-4.3cm}\caption{ Minimum number of PRNs and their placements for 2-LoS coverage. }\label{fig:double-coverage}
\vspace{-0.3cm}
\end{figure*}

 \begin{figure*}[!t]
\advance\leftskip+2.5cm
\includegraphics[width=13.1cm]{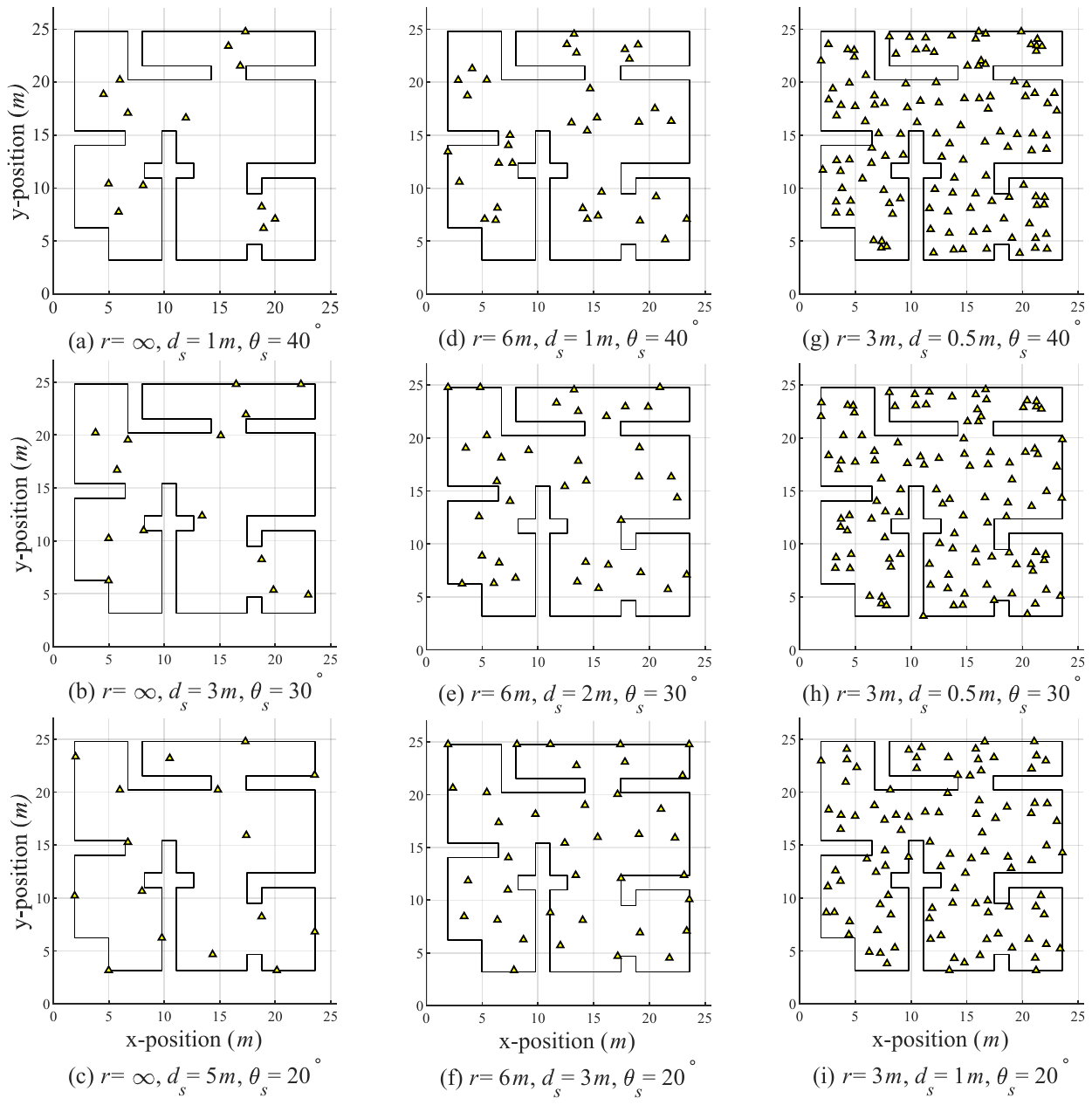}
    \vspace{-4.1cm}\caption{  Minimum number of PRNs and their placements for 3-LoS coverage. }\label{fig:triple-coverage}
\vspace{-0.3cm}
\end{figure*}

Using the secondary MCC method, Figure \ref{fig:double-coverage} illustrates the minimum number of PRNs and their placement, ensuring 2-LoS coverage.
 Setting $r=\infty$, Figs. \ref{fig:double-coverage}a-c) indicate the resulting $8$, $8$, and  $9$ PRNs for MSD values $d_s=1, 3$, and $5m$, respectively. 
 It is noted, for example, in Fig.  \ref{fig:double-coverage}c) that any UE within the layout has LoS to at least two PRNs located  at least five meters from each other.
 When $r=6m$ and using the same MSD values, Figs. \ref{fig:double-coverage}d-f) display the deployment of the resulting $18$, $20$, and $34$ PRNs.
Notably, in Fig. \ref{fig:double-coverage}f),  even though there are many pairs of PRNs closer than five meters, a UE anywhere in the layout falls within the LoS of at least a twin PRNs with distance of larger than $d_s=5m$.
 A similar pattern is observed in Fig. \ref{fig:double-coverage}g-i) that resulted in $60$, $64$, and $67$ PRNs, when $r=3m$ and the MSD values $d_s=0.5, 1$, and $2m$, respectively.
A higher MSD value may dramatically increase the number of PRNs required for 2-LoS coverage.

The placement of the minimum number of PRNs needed for 3-LoS coverage is illustrated in Fig. \ref{fig:triple-coverage} using the trinary MCC method. 
 Here, we set an MSA value in addition to the MSD. When $r=\infty$,  Figs.  \ref{fig:triple-coverage}a-c) show the resulting $13$, $14$, and $16$ PRNs and their locations when $(d_s, \theta_s)=(1m, 40^\circ)$, $(3m, 30^\circ)$, and $(5m, 20^\circ)$, respectively. As can be seen, a UE anywhere in the layout has  LoS conditions to at least a triplet of PRNs. By setting $r=6m$ and changing the MSD values to $d_s=1, 2$, and $3m$,  Figs. \ref{fig:triple-coverage}d-f) indicate the resulting $34$, $36$, and $37$ PRNs. Similarly,  Figs. \ref{fig:triple-coverage}g-i) depict the deployment of $122, 110$, and $99$ PRNs ensuring 3-LoS coverage when $r=3m$ and the MSD values $d_s=0.5, 0.5$, and $1m$, respectively. As imposed by MSD and MSA values, the three PRNs closest to a UE are not necessarily the triplet of PRNs serving the UE.

\begin{figure}[!t]
\centering
\advance\leftskip-1cm
\advance\rightskip-1cm
\includegraphics[width=10.5cm]{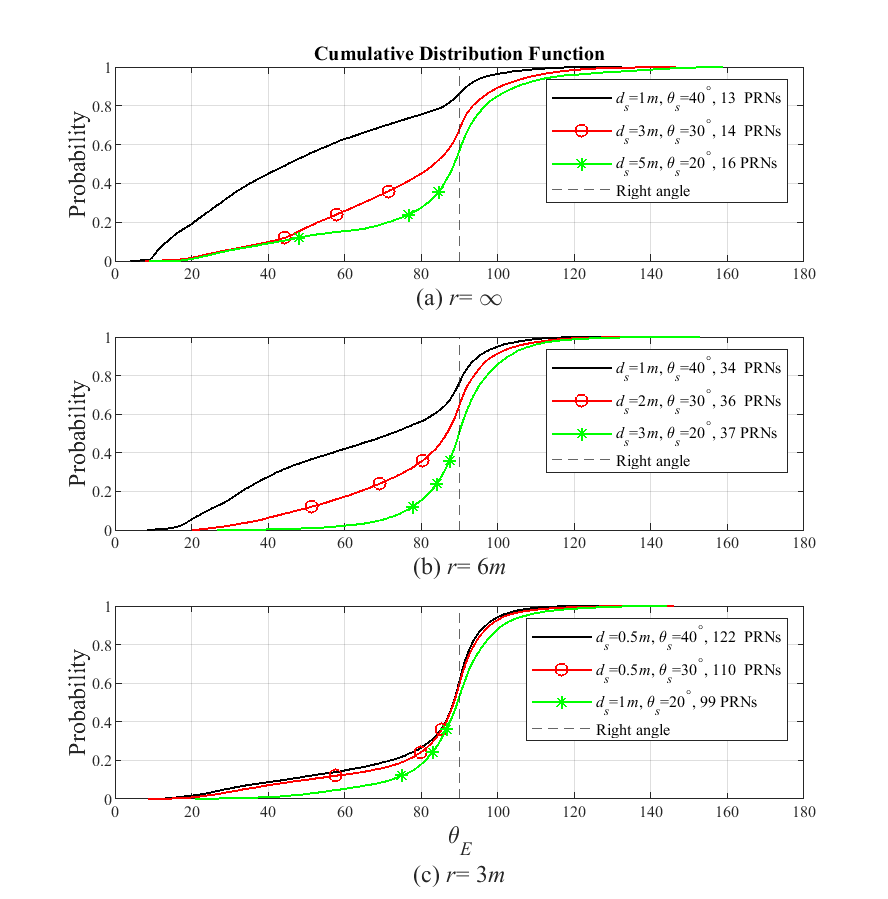} 
    \vspace{-0.4cm}\caption{ Comparison of EVA cumulative distribution for a) the PRN deployments in Figs. \ref{fig:triple-coverage}a-c), b) the PRN deployments in  Figs. \ref{fig:triple-coverage}d-f), and c) the PRN deployments in  Figs. \ref{fig:triple-coverage}g-i).}\label{fig:CDFs-triple-coverage}
\vspace{-0.2cm}
\end{figure}

 As a performance metric for precision of positioning, we evaluate the PRN deployments in Figure \ref{fig:triple-coverage} by calculating the EVA for UEs distributed uniformly within the sample layout.  
 Fig. \ref{fig:CDFs-triple-coverage}a)  compares the Cumulative Distribution Function CDF of EVAs for the PRNs placed as in Figs. \ref{fig:triple-coverage}a-c).
It is observed in Fig. \ref{fig:CDFs-triple-coverage}a) that deploying $13$ PRNs as in  Fig. \ref{fig:triple-coverage}a) results in only $40\%$ of the layout area having an EVA within the interval $|90^\circ-\theta_E|<30^\circ$, while the deployment of $16$ PRNs as in Fig. \ref{fig:triple-coverage}c) yields  $80\%$ of the layout within the interval $|90^\circ-\theta_E|<30^\circ$. Similarly, Fig. \ref{fig:CDFs-triple-coverage}b) shows that deploying $37$ PRNs, as in Fig. \ref{fig:triple-coverage}f), provides the EVA satisfying  $|90^\circ-\theta_E|<30^\circ$ for more than $95\%$ of the layout area.
Finally, Fig. \ref{fig:CDFs-triple-coverage}c) shows that with $r=3m$, the deployment of only $99$ PRNs gives even  better EVA distribution than $110$ and $122$ PRNs.

\color{black}{}

\pagestyle{empty}

\vspace{0mm}
\section{Conclusion}
\label{sec:5}

Line-of-sight visibility is essential for accurate indoor positioning.
We proposed three deployment algorithms to determine the minimum number of reference nodes and their precise placement to guarantee LoS visibility regardless of the UE location in the indoor service area. To do this, the floor plan was first modeled by a graph; then, the graph was partitioned into a minimum number of cliques whose LoS visibility area determined the locations of the reference nodes. 
 To improve the accuracy of the positioning methods based on triangulation (AoA/AoD) and trilateration (ToA/TDOA/RTT), we introduced two design parameters, i.e., a minimum separation distance and a minimum separation angle between reference nodes, to identify the most convenient placement of them. 
Furthermore, we considered different coverage ranges for reference nodes, which become relevant when moving from mmWave (indoor range-unconstrained) into THz and optical wireless (indoor range-constrained) related problems. As expected, we observed that the number of reference nodes required to guarantee LoS visibility grows fast when the range of the positioning nodes becomes shorter. We also observed that increasing the minimum separation distance between the reference nodes improves the positioning accuracy for large ranges.  In contrast, as the range decreases,  the minimum separation angle becomes more critical regarding the positioning accuracy and the number of required reference nodes.





 

\pagestyle{empty}

  \section*{APPENDIX A}
\textit{Proof for Theorem \ref{Theorem:effective-angle}}:
The first assumption is that $U$ falls anywhere within the interior of $q$ and therefore we have $\theta_{ij}+\theta_{ik}+\theta_{jk}=360^\circ$ with non-reflex visibility angles $\hat{Q}_k\leq \theta_{ij}\leq 180^\circ$, $\hat{Q}_j\leq \theta_{ik}\leq 180^\circ$, and $\hat{Q}_i\leq \theta_{jk}\leq 180^\circ$,  
see Fig. \ref{fig:angle_define}a).  Let $\theta_{jk}$ be the smallest visibility angle. So, we have $\theta_{jk}\leq 120^\circ$ alongside  $180^\circ\leq\theta_{jk}+\theta_{ij}$ and $180^\circ\leq\theta_{jk}+\theta_{ik}$. Then, following a brief analysis, it can be seen that $\theta_{jk}$ is the visibility angle closest to the right angle and as a consequence, $\theta_E=\theta_{jk}$.
 Accordingly,
\begin{equation}
\theta_{s}\leq \hat{Q}_i\leq\theta_{jk}=\theta_E \leq 120^\circ. 
\label{eq:appendix0}    
\end{equation}Then, from the feasibility interval $0^\circ \leq \theta_s \leq 60^\circ$ and  (\ref{eq:appendix0}), we conclude (\ref{eq1}).

The second assumption is that $U$ lies outside $q$.
The exterior is divided into four regions in Fig. \ref{fig:cases1} by sketching the three sidelong circles and the circumscribed circle in black and red, respectively. Sidelong circles take the sides of a polygon as diameters. We determine the bounds for EVA  at each region.

\begin{itemize}
    \item Region I, shaded in Fig. \ref{fig:cases1}a), represents the area lying outside the sidelong circles. Let $U$ be located within this region. Then, all the  visibility angles are acute and therefore $\theta_E=\max(\theta_{ij},\theta_{ik}, \theta_{jk})\leq90^\circ$.
     We define $\theta_l$  as the angle between two tangent lines through $U$ to the inscribed circle of $q$.
     $\theta_l$ represents a lower bound to $\theta_E$, see Fig. \ref{fig:incircle1}a). 
    This fact motivates to establish a lower bound to $\theta_l$ in terms of $d_s$ and $\theta_s$.
    As in Fig. \ref{fig:incircle1}b), we consider the right triangle $UO_0H$, where $O_0$ is the center point of the inscribed circle with radius $r_0$ and $H$ denotes one of two tangency points from $U$. We have 
  \begin{equation}
\tan{(\dfrac{\theta_l}{2})}=\tan{( \widehat{ HUO_0})}=\dfrac{\|O_0H\|}{\|UH\|}\geq\dfrac{r_0}{r},
\label{eq:appendix1}    
\end{equation}
where the inequality is derived from  the facts $\|O_0H\|=r_0$ and $\|UH\|\leq r$. The latter is because the UE is located within the maximum range of the  PRNs and $H$ falls inside $q$. Here, we also used the fact that the segment $UO_o$ bisects $\theta_l$. Without loss of generality, let $\hat{Q}_k$ be the smallest inner angle of $q$. To find a lower-bound to $r_0$, we consider the right triangle $Q_kO_0L$, where $L$ denotes the tangency point of the side $Q_iQ_k$. Since the PRNs form a triplet, we have $\theta_s\leq \hat{Q}_k$ and $d_s\leq \|Q_iQ_k\|$. Thus,
\begin{equation}
\tan{(\dfrac{\theta_s}{2})}\leq\tan{( \widehat{LQ_kO_0})}=\dfrac{\|O_0L\|}{\|Q_kL\|}\leq\dfrac{r_0}{d_s/2},
\label{eq:appendix2}    
\end{equation}
as the tangent is an increasing function, $\|O_0L\|=r_0$, and $d_s/2\leq \|Q_kL\|$. The last term is because $\hat{Q}_k$  is the smallest angle of $q$ and therefore the segment $Q_kL$
 has a longer side portion than $Q_iL$ in  $Q_iQ_k$.
 Then, from (\ref{eq:appendix1}), (\ref{eq:appendix2}), and  $\theta_l\leq\theta_E\leq 90^\circ$, the EVA of $U$ located anywhere inside region I is confined to
    \begin{equation}
2\times\arctan{\big(\dfrac{d_s}{2r} \tan{(\dfrac{\theta_s}{2})}\big)}\leq \theta_E\leq 90^\circ.
\label{eq:appendix3}    
\end{equation}

\item 
Region II, shaded in Fig. \ref{fig:cases1}b), represents the locus that lie inside only one of the sidelong circles. Without loss of generality, let $U$ fall within the sidelong circle of $Q_iQ_k$ and therefore we have $90^\circ\leq\theta_{ik}=\theta_{ij}+\theta_{jk}\leq 180^\circ$ with $\theta_{ij},\theta_{jk}\leq 90^\circ$. Then, we have either $\theta_E=\max(\theta_{ij},\theta_{jk})$ or $\theta_E=\theta_{ik}$ depending on which one is closest to $90^\circ$. So, a brief analysis indicates that the EVA is restricted to 
  \begin{equation}
60^\circ\leq\theta_E \leq 120^\circ, 
\label{eq:appendix4}    
\end{equation}
for a UE anywhere inside region II.

\begin{figure}[!t]
\centering
\hspace{0cm}\includegraphics[width=7.5cm]{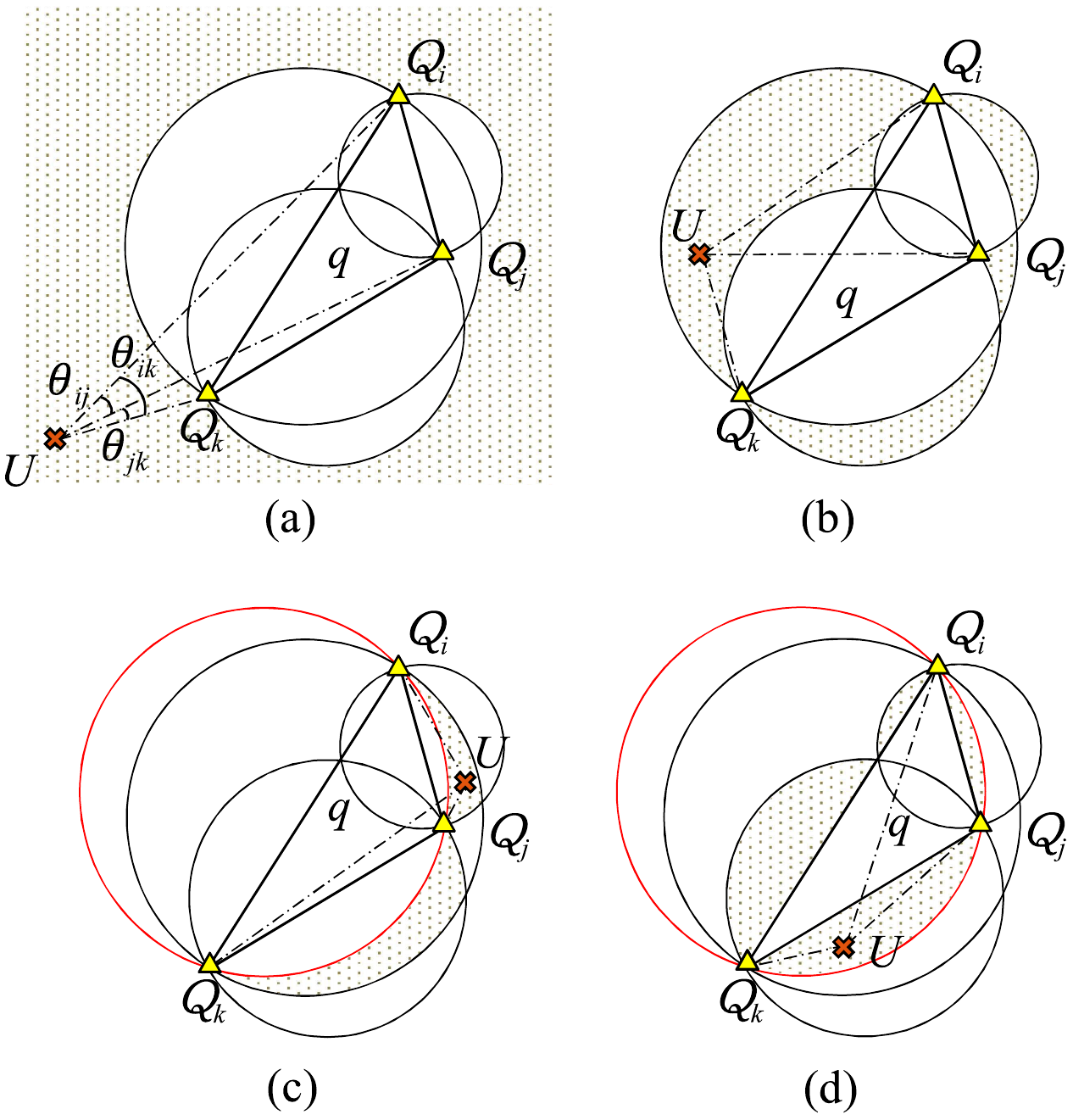}
    \vspace{-2.0cm}\caption{  A UE lying on the exterior of the triangle $q$}\label{fig:cases1}
\vspace{-0.3 cm}
\end{figure}

\item Region III, the shaded area of Fig. \ref{fig:cases1}c) represents the area lying within exactly two sidelong circles and outside the circumscribed circle.
 Without loss of generality, let $U$ be within the sidelong circles of $Q_iQ_j$ and $Q_iQ_k$. Then, we have $\theta_{ik}+\theta_{jk}=\theta_{ij}$ with $90^\circ \leq \theta_{ik}, \theta_{ij}\leq 180^\circ$. A brief analysis here shows that $\theta_E=\theta_{ik}$. Besides, since $U$ lies  outside the circumscribed circle, we have $\theta_{ik}\leq \hat{Q}_j\leq 180^\circ-2\theta_s$, wherein the right inequality shows the maximum inner angle of $q$.  So, the  EVA is limited to
\begin{equation}
90^\circ\leq\theta_E \leq 180^\circ-2\theta_s, 
\label{eq:appendix5}    
\end{equation}
for a UE lying inside region III.

\item  Region IV, shaded in Fig. \ref{fig:cases1}d),  indicates the area enclosed by exactly two sidelong circles and the interior of the circumscribed circle.
 Let $U$ be inside the sidelong circles $Q_jQ_k$ and $Q_iQ_k$. Then, we have  $90^\circ\leq\theta_{ij}+\theta_{ik}=\theta_{jk}\leq 180^\circ$ with $\theta_{ij}\leq90^\circ \leq \theta_{ik}$. So, the EVA here can be either $\theta_E=\theta_{ij}$ or $\theta_E=\theta_{ik}$. 
 Since $U$ falls inside the circumscribed circle the PRNs from a triplet, we have $\theta_s\leq \hat{Q}_k\leq\theta_{ij}$ and therefore  $90^\circ\leq\theta_{ik}\leq 180-\theta_s$. Thus, a brief analysis shows that the EVA is bounded   by
\begin{equation}
|90^\circ-\theta_E|\leq |90^\circ-\theta_s|, 
\label{eq:appendix6}    
\end{equation}
for a UE falling inside region IV. 
\end{itemize}
 The points inside all three sidelong circles fall inside $q$. So, we skip this case.
Finally, from the feasible intervals of MSD and MSA values, the union of (\ref{eq:appendix3})-(\ref{eq:appendix6}) can be formed as 
(\ref{eq:effective_LoS_angle_outside}).

\begin{figure}[!t]
\centering
\hspace{0cm}\includegraphics[width=8.1cm]{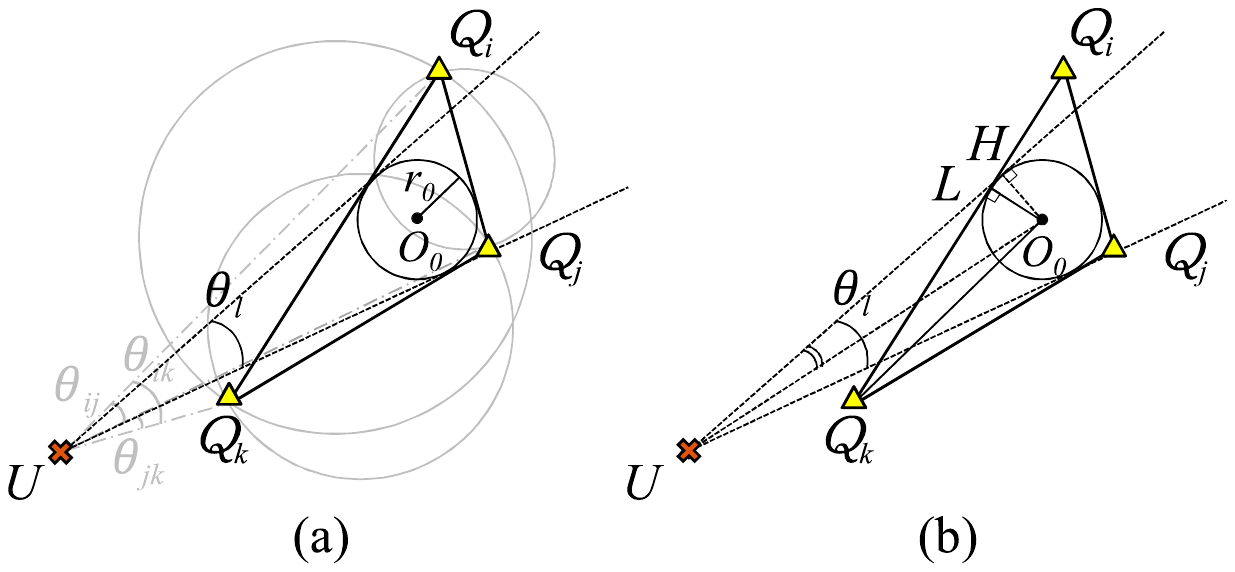}
    \vspace{-7.1cm}\caption{  A UE lying on the exterior of $q$ and the sidelong circles}\label{fig:incircle1}
\vspace{-0.3 cm}
\end{figure}



\vspace{0mm}

\bibliography{main}
\bibliographystyle{ieeetr}

\end{document}